\documentclass{amsart}
\usepackage{amsmath}
\usepackage{amsfonts}

\newtheorem{theorem}{Theorem}
\theoremstyle{plain}

\newtheorem{corollary}{Corollary}

\newtheorem{definition}{Definition}

\newtheorem{lemma}{Lemma}

\newtheorem{proposition}{Proposition}
\newtheorem{remark}{Remark}

\DeclareMathOperator{\Div}{div}
 \numberwithin{equation}{section}
\input{tcilatex}

\begin{document}
\title[Homogenization of Ladyzhenskaya equations]{Almost periodic
homogenization of a generalized Ladyzhenskaya model for incompressible
viscous flow}
\author{Hermann Douanla}
\address{H. Douanla, Department of Mathematical Sciences, Chalmers
University of Technology, Gothenburg, SE-41296, Sweden}
\email{douanla@chalmers.se}
\author{Jean Louis Woukeng}
\address{J.L. Woukeng, Department of Mathematics and Computer Science,
University of Dschang, P.O. Box 67, Dschang, Cameroon}
\email{jwoukeng@yahoo.fr}
\date{}

\begin{abstract}
The paper deals with the existence and almost periodic homogenization of
some model of generalized Navier-Stokes equations. We first establish an
existence result for non-stationary Ladyzhenskaya equations with a given non
constant density and an external force depending nonlinearly on the
velocity. Next, the density of the fluid being non constant, we combine some
compactness arguments with the sigma-convergence method to study the
asymptotic behavior of the velocity field.
\end{abstract}

\maketitle

\section{Introduction}

The Navier-Stokes equations model the motion of Newtonian fluids. In order
to understand the phenomenon of turbulence related to the motion of a fluid,
several mathematical models have been developed and studied over the years.
We refer to, e.g. \cite{BBN, He, Gunz, Ladyz, Lions}, just to cite a few.

To investigate the turbulence in non-Newtonian fluid, we consider a model
close to the Ladyzhenskaya one \cite{Ladyz}. In this model, the viscosity
and the density are non constant. To be more precise, let $\varepsilon>0$ be
a small parameter representing the scale of the inhomogeneities. The
equation of the motion reads 
\begin{equation}
\rho ^{\varepsilon }\frac{\partial \boldsymbol{u}_{\varepsilon }}{\partial t}%
-\Div\left( a^{\varepsilon }\nabla \boldsymbol{u}_{\varepsilon
}+b^{\varepsilon }\left\vert \nabla \boldsymbol{u}_{\varepsilon }\right\vert
^{p-2}\nabla \boldsymbol{u}_{\varepsilon }\right) +(\boldsymbol{u}%
_{\varepsilon }\cdot \nabla )\boldsymbol{u}_{\varepsilon }+\nabla
q_{\varepsilon }=\rho ^{\varepsilon }f^{\varepsilon }(\cdot ,\boldsymbol{u}%
_{\varepsilon })\text{ in }Q_{T}  \label{1}
\end{equation}%
where $\rho $ is the non constant positive known density which is bounded
from above and from below away from zero, $\boldsymbol{u}_{\varepsilon }$
and $q_{\varepsilon }$ are the unknown velocity and pressure, respectively.
The viscosities $a=(a_{ij})_{1\leq i,j\leq N}$ and $b$ both depend on the
spatial and time variables. The matrix $a$ is coercive whereas the function $%
b$ is positive, bounded from above and from below away from zero. The
density and the viscosities are scaled as follows: $\rho ^{\varepsilon }(x)
=\rho \left( \frac{x}{\varepsilon }\right) ,\ a_{ij}^{\varepsilon
}(x,t)=a_{ij}\left( \frac{x}{\varepsilon },\frac{t}{ \varepsilon ^{2}}%
\right) \text{ and }b^{\varepsilon }(x,t) =b\left( \frac{x}{\varepsilon },%
\frac{t}{\varepsilon ^{2}}\right)$ whereas the external force $%
f^{\varepsilon }(\cdot,\boldsymbol{u} _{\varepsilon })$ which is a Lipschitz
function of the velocity $\boldsymbol{u}_{\varepsilon }$ is scaled as
follows $f^{\varepsilon }(\cdot ,r)(t)=f\left( \frac{t}{\varepsilon ^{2}}%
,r\right)\ \text{ for }(x,t)\in Q_{T}$ and $r\in\mathbb{R}^N$. The problem
is stated in details in the beginning of section 2.

The equation (\ref{1}) models various types of motion of non-Newtonian
fluids. We cite a few examples. If $a=(\nu _{0}\delta _{ij})_{1\leq i,j\leq
N}$ ($\delta _{ij}$ the Kronecker delta) and $b=\nu _{1}$ where $\nu _{0}$
and $\nu _{1}$ are positive constants, then we get the Ladyzhenskaya
equations. In that case, the analysis conducted in \cite{Gunz} reveals that
one may either let $\nu _{1}\rightarrow 0$ (and get the usual Navier-Stokes
equations) or let $\nu _{0}\rightarrow 0$, and hence we are led to the
power-law fluids equations. In particular when $p=3$, we get the
Smagorinski's model of turbulence \cite{SM} with $\nu _{0}$ being the
molecular viscosity and $\nu _{1}$ the turbulent viscosity. Another model
included in (\ref{1}) is the equation of incompressible bipolar fluids \cite%
{BBN}.

There are several works dealing with the Navier-Stokes equations with
external force depending on the velocity. These types of equations are
commonly known as the \textit{generalized} Navier-Stokes equations. This
work is different from the previous ones in the sense that it combines both
a non-constant density, the non-Newtonian fluid effect, and the external
force depending on the velocity. It should be noted that if in equation (\ref%
{1}) we replace the gradient $\nabla \boldsymbol{u}$ by its symmetric part $%
\frac{1}{2}(\nabla \boldsymbol{u}+\nabla ^{T}\boldsymbol{u})$, then thanks
to the Korn's inequality, the mathematical analysis does not change,
although the model becomes in this case, physical.

Started in the 70's the mathematical theory of Homogenization is nowadays
divided into two major components: the \emph{individual} homogenization
theory (also known as the \emph{deterministic} homogenization theory) and
the \emph{random} homogenization theory (also known as the \emph{stochastic}
homogenization theory). In this paper we are concerned with a special case
of individual homogenization theory, namely the almost periodic one. More
precisely, we assume throughout the paper that the density $\rho $, the
viscosities $a$ and $b$ and the source term $f$ are almost periodic
functions as specified in the beginning of Section 4. It should be noted
that we could consider a more general deterministic setting with
coefficients satisfying general deterministic assumptions covering a large
set of concrete behaviors (see e.g., \cite{Douanla1, Hom1, CMP}) such as the
periodic one, the almost periodic one, the convergence at infinity and many
more besides. We only deal with almost periodic homogenization for the sake
of simplicity. Our results carry over mutatis mutandis to the general
deterministic setting. Concerning the homogenization of (\ref{1}), it should
also be stressed that this is the first time that such an analysis is
conducted beyond the periodic setting.

The goal of this paper is twofold. We first establish an existence result
for (\ref{1}) and then perform the homogenization process for (\ref{1}) in
the almost periodic setting. The space dimension is either $2$ or $3$. The
only place where the analysis changes depending on the dimension is the
proof of the inequality (\ref{0.7}) which is very useful in the proofs of
the existence result, theorem \ref{t1.1} and the compactness result,
Proposition \ref{p2.2}. We stress that the existence result Theorem \ref%
{t1.1} is not solving the millennium problem of the existence and smoothness
of the Navier-Stokes equations as stated by Fefferman \cite{Fefferman} since
we only consider weak solutions here with no regularity result.

The paper is organized as follows. Section 2 deals with the complete
statement of the problem, the proof of the existence result and some a
priori estimates. In Section 3 we gather some necessary tools about the $%
\Sigma $-convergence method (which is just the appropriate generalization of
the well-known two-scale convergence method) in the algebra of continuous
almost periodic functions. Finally in Section 4, we state and prove the
homogenization result.

Throughout Section 3, vector spaces are assumed to be complex vector spaces,
and scalar functions are assumed to take complex values. We shall always
assume that the numerical space $\mathbb{R}^{m}$ (integer $m\geq 1$) and its
open sets are each equipped with the Lebesgue measure $dx=dx_{1}...dx_{m}$.

\section{Statement of the problem: existence result and a priori estimates}

\subsection{Problem setting: existence result}

We consider $N$-dimensional problem, $N=2,3$. In what follows, all the
function spaces are real-valued spaces and scalar functions assume real
values.

Let $1+\frac{2N}{N+2}\leq p<\infty $, and $T>0$ a real number. Let $%
Q_{T}=Q\times \left( 0,T\right) $ where $Q$ is a bounded smooth domain in $%
\mathbb{R}^{N}$. We consider the following well-known spaces \cite{Lions,
Temam}: $\mathcal{V}=\{\boldsymbol{\varphi }\in \mathcal{C}_{0}^{\infty
}(Q)^{N}:\Div\boldsymbol{\varphi }=0\}$; $\mathbb{V}=$ closure of $\mathcal{V%
}$ in $W^{1,p}(Q)^{N}$; $\mathbb{H}=$ closure of $\mathcal{V}$ in $%
L^{2}(Q)^{N}$. In view of the smoothness of $Q$, it is known that $\mathbb{V}%
=\{\boldsymbol{u}\in W_{0}^{1,p}(Q)^{N}:\Div\boldsymbol{u}=0\}$ and $\mathbb{%
H}=\{\boldsymbol{u}\in L^{2}(Q)^{N}:\Div\boldsymbol{u}=0$ and $\left. 
\boldsymbol{u}\right\vert _{\partial Q}\cdot n=0\}$, where $\left. 
\boldsymbol{u}\right\vert _{\partial Q}$ denotes the trace of $\boldsymbol{u}
$ on $\partial Q$ and $n$ is the outward unit vector normal to $\partial Q$.
The space $\mathbb{V}$ is hence endowed with the gradient norm and $\mathbb{H%
}$ with the $L^{2}$-norm. For the sake of simplicity, we denote the
respective norms of $\mathbb{V}$ and $\mathbb{H}$ by $\left\Vert \cdot
\right\Vert $ and $\left\vert \cdot \right\vert $. We denote by $\left(
,\right) $ the inner product in $\mathbb{H}$, as well as the scalar product
in $\mathbb{R}^{N}$. The associated Euclidean norm in $\mathbb{R}^{N}$ is
also denoted by $\left\vert \cdot \right\vert $. All duality pairing will be
denoted by $\left\langle ,\right\rangle $ without any specification of the
spaces involved.

With all this in mind, we shall consider the functions $a$, $b$, $\rho $ and 
$f$ constrained as follows.

\begin{itemize}
\item[(\textbf{A1})] The matrix $a=(a_{ij})_{1\leq i,j\leq N}\in L^{\infty }(%
\mathbb{R}^{N+1})^{N\times N}$ satisfies $a_{ij}=a_{ji}$ and 
\begin{equation*}
(a(y,\tau )\lambda )\cdot \lambda \geq \nu _{0}\left\vert \lambda
\right\vert ^{2}\text{ for all }\lambda \in \mathbb{R}^{N}\text{ and a.e. }%
(y,\tau )\in \mathbb{R}^{N+1}
\end{equation*}%
for some positive $\nu _{0}$.

\item[(\textbf{A2})] The function $b\in L^{\infty }(\mathbb{R}^{N+1})$ is
such that $\nu _{1}\leq b\leq \nu _{2}$ for some positive $\nu _{1}$ and $%
\nu _{2}$.

\item[(\textbf{A3})] The density function $\rho$ belongs to $\in L^{\infty }(%
\mathbb{R}^{N})$ and satisfies $\Lambda ^{-1}\leq \rho (y)\leq \Lambda $
a.e. $y\in \mathbb{R}^{N}$, for some positive $\Lambda $.

\item[(\textbf{A4})] The mapping $f:\mathbb{R}\times \mathbb{R}%
^{N}\rightarrow \mathbb{R}^{N}$, $(t,r)\mapsto f(t,r)$ is continuous with
the following properties:

\begin{itemize}
\item[(i)] $f$ maps continuously $\mathbb{R}\times \mathbb{H}$ into $\mathbb{%
H}$ ($f(\tau ,u)\in \mathbb{H}$ for any $u\in \mathbb{H}$)

\item[(ii)] There is a positive constant $k$ such that 
\begin{equation*}
\begin{array}{l}
\left\vert f(\tau ,0)\right\vert \leq k\text{ for all }\tau \in \mathbb{R}
\\ 
\left\vert f(\tau ,r_{1})-f(\tau ,r_{2})\right\vert \leq k\left\vert
r_{1}-r_{2}\right\vert \text{ for all }r_{1},r_{2}\in \mathbb{R}^{N}\text{
and }\tau \in \mathbb{R}\text{.}%
\end{array}%
\end{equation*}
\end{itemize}
\end{itemize}

\noindent We deduce from [part (ii) of] (\textbf{A4}) the existence of a
positive constant $c$ such that $\left\vert f(\tau ,r)\right\vert \leq
c(1+\left\vert r\right\vert )$ for any $(\tau ,r)\in \mathbb{R}\times 
\mathbb{R}^{N}$. The scaled functions $\rho ^{\varepsilon }\in L^{\infty
}(Q) $, $a^{\varepsilon }=(a_{ij}^{\varepsilon })_{1\leq i,j\leq N}\in
L^{\infty }(Q_{T})^{N\times N}$, $f^{\varepsilon }(\cdot ,r)\in \mathcal{C}%
(0,T)$ (for any $r\in \mathbb{R}^{N}$) and $b^{\varepsilon }\in L^{\infty
}(Q_{T})$ are defined as follows: 
\begin{eqnarray*}
\rho ^{\varepsilon }(x) &=&\rho \left( \frac{x}{\varepsilon }\right) ,\
a_{ij}^{\varepsilon }(x,t)=a_{ij}\left( \frac{x}{\varepsilon },\frac{t}{%
\varepsilon ^{2}}\right) ,\ f^{\varepsilon }(\cdot ,r)(t)=f\left( \frac{t}{%
\varepsilon ^{2}},r\right) \text{ and } \\
b^{\varepsilon }(x,t) &=&b\left( \frac{x}{\varepsilon },\frac{t}{\varepsilon
^{2}}\right) \text{ for }(x,t)\in Q_{T}.
\end{eqnarray*}%
Let $\varepsilon>0$ be a small parameter. We will make use of the following
notation 
\begin{equation*}
\left( \boldsymbol{u},\boldsymbol{v}\right)_{\varepsilon } =\int_{Q}\rho
^{\varepsilon }(x)\boldsymbol{u}(x)\cdot \boldsymbol{v}(x)dx\ \ \ (%
\boldsymbol{u},\boldsymbol{v}\in \mathbb{H}),
\end{equation*}%
which defines an inner product in $\mathbb{H}$ making it a Hilbert space.
The associated norm is denoted in the sequel by $| \cdot |_{\varepsilon }$
and is obviously equivalent to the natural norm of $\mathbb{H}$ denoted by $%
|\cdot|$.

%

Given $u_0\in \mathbb{H}$, we are interested in the asymptotic behavior of
the sequence of velocity field $(\boldsymbol{u}_{\varepsilon })_{\varepsilon
>0}$ of the following generalized Ladyzhenskaya equation 
\begin{equation}
\begin{array}{l}
\rho ^{\varepsilon }\frac{\partial \boldsymbol{u}_{\varepsilon }}{\partial t}%
-\Div\left( a^{\varepsilon }\nabla \boldsymbol{u}_{\varepsilon
}+b^{\varepsilon }\left\vert \nabla \boldsymbol{u}_{\varepsilon }\right\vert
^{p-2}\nabla \boldsymbol{u}_{\varepsilon }\right) +(\boldsymbol{u}%
_{\varepsilon }\cdot \nabla )\boldsymbol{u}_{\varepsilon }+\nabla
q_{\varepsilon }=\rho ^{\varepsilon }f^{\varepsilon }(\cdot ,\boldsymbol{u}%
_{\varepsilon })\text{ in }Q_{T} \\ 
\Div\boldsymbol{u}_{\varepsilon }=0\text{ in }Q_{T} \\ 
\boldsymbol{u}_{\varepsilon }=0\text{ on }\partial Q\times (0,T) \\ 
\boldsymbol{u}_{\varepsilon }(x,0)=\boldsymbol{u}^{_{0}}(x)\text{ in }Q.%
\end{array}
\label{0.1}
\end{equation}%
We introduce the functionals (for fixed $\varepsilon >0$) 
\begin{equation*}
a_{I}^{\varepsilon }(t;\boldsymbol{u},\boldsymbol{v})=\int_{Q}\left(
a^{\varepsilon }\nabla \boldsymbol{u}\right) \cdot \nabla \boldsymbol{v}%
dx+\int_{Q}b^{\varepsilon }\left\vert \nabla \boldsymbol{u}\right\vert
^{p-2}\nabla \boldsymbol{u}\cdot \nabla \boldsymbol{v}dx
\end{equation*}%
and 
\begin{equation*}
b_{I}(\boldsymbol{u},\boldsymbol{v},\boldsymbol{w})=\sum_{i,j=1}^{N}%
\int_{Q}u_{i}\frac{\partial v_{j}}{\partial x_{i}}w_{j}dx
\end{equation*}%
defined for $\boldsymbol{u},\boldsymbol{v},\boldsymbol{w}\in
W_{0}^{1,p}(Q)^{N}$. Then the following estimates hold: 
\begin{equation}
\left\vert a_{I}^{\varepsilon }(t;\boldsymbol{u},\boldsymbol{v})\right\vert
\leq \left\Vert a\right\Vert _{\infty }\left\Vert \nabla \boldsymbol{u}%
\right\Vert _{L^{2}(Q)}\left\Vert \nabla \boldsymbol{v}\right\Vert
_{L^{2}(Q)}+\nu _{2}\left\Vert \nabla \boldsymbol{u}\right\Vert
_{L^{p}(Q)}^{p-1}\left\Vert \nabla \boldsymbol{v}\right\Vert _{L^{p}(Q)},
\label{0.2}
\end{equation}%
\begin{equation}
a_{I}^{\varepsilon }(t;\boldsymbol{v},\boldsymbol{v})\geq \nu _{0}\left\Vert
\nabla \boldsymbol{v}\right\Vert _{L^{2}(Q)}^{2}+\nu _{1}\left\Vert \nabla 
\boldsymbol{v}\right\Vert _{L^{p}(Q)}^{p}  \label{0.3}
\end{equation}%
for all $\boldsymbol{u},\boldsymbol{v}\in W_{0}^{1,p}(Q)^{N}$. Since $p\geq
1+\frac{2N}{N+2}$ (hence $p\geq 2$) we have $L^{p}(Q)\hookrightarrow
L^{2}(Q) $ so that, by the estimate (\ref{0.2}) we infer the existence of an
operator $\mathcal{A}^{\varepsilon }(t):\mathbb{V} \rightarrow \mathbb{V}%
^{\prime }$ such that 
\begin{equation*}
a_{I}^{\varepsilon }(t;\boldsymbol{u},\boldsymbol{v})=\left\langle \mathcal{A%
}^{\varepsilon }(t)\boldsymbol{u},\boldsymbol{v}\right\rangle \text{ for all 
}\boldsymbol{u},\boldsymbol{v}\in \mathbb{V}\text{.}
\end{equation*}%
Hence the existence of a bounded operator $\mathcal{A}^{\varepsilon
}:L^{p}(0,T;\mathbb{V})\rightarrow L^{p^{\prime }}(0,T;\mathbb{V}^{\prime })$
such that 
\begin{equation*}
\left\langle \mathcal{A}^{\varepsilon }\boldsymbol{u},\boldsymbol{v}%
\right\rangle =\int_{0}^{T}\left\langle \mathcal{A}^{\varepsilon }(t)%
\boldsymbol{u}(t),\boldsymbol{v}(t)\right\rangle dt\text{ for all }%
\boldsymbol{u},\boldsymbol{v}\in L^{p}(0,T;\mathbb{V}).
\end{equation*}

Now, dealing with the trilinear functional $b_{I}$, we have that \cite%
{Lions, Temam}%
\begin{eqnarray*}
b_{I}(\boldsymbol{u},\boldsymbol{v},\boldsymbol{v}) &=&0\text{ for all }%
\boldsymbol{u}\in \mathbb{V}\text{ and }\boldsymbol{v}\in W_{0}^{1,p}(Q)^{N}
\\
b_{I}(\boldsymbol{u},\boldsymbol{u},\boldsymbol{v}) &=&-b_{I}(\boldsymbol{u},%
\boldsymbol{v},\boldsymbol{u})\text{ for all }\boldsymbol{u}\in \mathbb{V}%
\text{ and }\boldsymbol{v}\in W_{0}^{1,p}(Q)^{N}\text{.}
\end{eqnarray*}%
The following Sobolev embedding holds true \cite{Adams}: $%
W^{1,p}(Q)\hookrightarrow L^{r}(Q)$ for $\frac{1}{r}=\frac{1}{p}-\frac{1}{N}$
if $\frac{1}{p}-\frac{1}{N}>0$, and for any $r>1$ if $\frac{1}{p}-\frac{1}{N}%
\leq 0$. In view of the choice of $p$ and $N$, we may always find $r>1$ such
that $\frac{2}{r}+\frac{1}{p}=1$, and hence, using H\"{o}lder's inequality,
we get 
\begin{equation*}
\left\vert b_{I}(\boldsymbol{u},\boldsymbol{u},\boldsymbol{v})\right\vert
=\left\vert b_{I}(\boldsymbol{u},\boldsymbol{v},\boldsymbol{u})\right\vert
\leq c\left\Vert \boldsymbol{u}\right\Vert _{L^{r}(Q)}^{2}\left\Vert \nabla 
\boldsymbol{v}\right\Vert _{L^{p}(Q)}\text{\ (}\boldsymbol{u},\boldsymbol{v}%
\in \mathbb{V}\text{).}
\end{equation*}%
We infer from the above inequality the existence of an element $B(%
\boldsymbol{u})\in \mathbb{V}^{\prime }$ such that 
\begin{equation*}
\left\langle B(\boldsymbol{u}),\boldsymbol{v}\right\rangle =b_{I}(%
\boldsymbol{u},\boldsymbol{u},\boldsymbol{v})\text{ for all }\boldsymbol{u},%
\boldsymbol{v}\in \mathbb{V}\text{.}
\end{equation*}%
This defines a bounded operator $B:\mathbb{V}\rightarrow \mathbb{V}^{\prime
} $ verifying the further property that, if $\boldsymbol{u}\in L^{p}(0,T;%
\mathbb{V})$ then $B(\boldsymbol{u})\in L^{p^{\prime }}(0,T;\mathbb{V}%
^{\prime })$. Indeed taking $\boldsymbol{u}\in L^{p}(0,T;\mathbb{V})$ we
have by the H\"{o}lder's inequality 
\begin{equation}
\left\Vert B(\boldsymbol{u})\right\Vert _{L^{p^{\prime }}(0,T;\mathbb{V}%
^{\prime })}\leq \left( \int_{0}^{T}\left\Vert \boldsymbol{u}(t)\right\Vert
_{L^{r}(Q)}^{2p^{\prime }}dt\right) ^{1/p^{\prime }}.  \label{0.4'}
\end{equation}%
But since $W^{1,p}(Q)\hookrightarrow L^{r}(Q)$, there is a positive constant 
$c_{0}$ independent of $\boldsymbol{u}$ such that 
\begin{equation*}
\left\Vert B(\boldsymbol{u})\right\Vert _{L^{p^{\prime }}(0,T;\mathbb{V}%
^{\prime })}\leq c_{0}\left( \int_{0}^{T}\left\Vert \boldsymbol{u}%
(t)\right\Vert ^{2p^{\prime }}dt\right) ^{1/p^{\prime }}
\end{equation*}%
where we recall that $\left\Vert \cdot \right\Vert $ stands for the norm in $%
\mathbb{V}$. We distinguish two situations: the case when $p\geq 3$ and the
case when $1+\frac{2N}{N+2}\leq p<3$. If $p\geq 3$, then $2p^{\prime }\leq p$%
, hence the H\"{o}lder's inequality again gives \ (noting that $\frac{p}{%
2p^{\prime }}\geq 1$) 
\begin{equation*}
\left( \int_{0}^{T}\left\Vert \boldsymbol{u}(t)\right\Vert ^{2p^{\prime
}}dt\right) ^{1/p^{\prime }}\leq c_{1}\left( \int_{0}^{T}\left\Vert 
\boldsymbol{u}(t)\right\Vert ^{p}dt\right) ^{2/p},
\end{equation*}%
so that 
\begin{equation}
\left\Vert B(\boldsymbol{u})\right\Vert _{L^{p^{\prime }}(0,T;\mathbb{V}%
^{\prime })}\leq c_{2}\left\Vert \boldsymbol{u}\right\Vert _{L^{p}(0,T;%
\mathbb{V})}^{2}.  \label{0.4}
\end{equation}%
Now, if $1+\frac{2N}{N+2}\leq p<3$, then it is easy to get an inequality
similar to (\ref{0.4}) for $N=2$. So we shall only deal with the case $N=3$.
With this in mind, returning to the inequality (\ref{0.4'}) which is true
for $r=2p^{\prime }$, we have 
\begin{equation*}
\left\Vert B(\boldsymbol{u})\right\Vert _{L^{p^{\prime }}(0,T;\mathbb{V}%
^{\prime })}\leq c_{2}\left( \int_{0}^{T}\left\Vert \boldsymbol{u}%
(t)\right\Vert _{L^{2p^{\prime }}(Q)}^{2p^{\prime }}dt\right) ^{2/p}.
\end{equation*}%
But \cite{Ladyz} (see also \cite[Lemma 4.1]{Gunz}) 
\begin{equation}
\left\Vert u\right\Vert _{L^{2p^{\prime }}(Q)}\leq C\left\Vert u\right\Vert
_{W^{1,p}(Q)}^{\alpha }\left\Vert u\right\Vert _{L^{2}(Q)}^{1-\alpha }\text{
for any }u\in W^{1,p}(Q)  \label{0.5'}
\end{equation}%
where $\alpha =3/(5p-6)$. But if $\boldsymbol{u}\in L^{p}(0,T;\mathbb{V}%
)\cap L^{\infty }(0,T;\mathbb{H})$, the function $t\mapsto \left\Vert 
\boldsymbol{u}(t)\right\Vert ^{2\alpha }\left\vert \boldsymbol{u}%
(t)\right\vert ^{2-2\alpha }$ lies in $L^{r_{0}}(0,T)$ for $r_{0}\geq 1$
satisfying $\frac{1}{r_{0}}=\frac{2\alpha }{p}+\frac{1}{\infty }$, and more
generally, in any $L^{r}(0,T)$ for $1<r\leq \frac{p}{2\alpha }$. Since $%
p^{\prime }\leq \frac{p}{2\alpha }$ (for $3>p\geq \frac{11}{5}$) we have
that it belongs to $L^{p^{\prime }}(0,T)$. Using H\"{o}lder's inequality in (%
\ref{0.5'}) (with exponent $r=\frac{p}{2\alpha p^{\prime }}$) we have 
\begin{equation*}
\left\Vert B(\boldsymbol{u})\right\Vert _{L^{p^{\prime }}(0,T;\mathbb{V}%
^{\prime })}^{p^{\prime }}\leq C\left( \int_{0}^{T}\left\Vert \boldsymbol{u}%
(t)\right\Vert ^{p}dt\right) ^{\frac{2\alpha p^{\prime }}{p}}\left(
\int_{0}^{T}\left\vert \boldsymbol{u}(t)\right\vert ^{(2-2\alpha )p^{\prime
}r^{\prime }}dt\right) ^{\frac{1}{r^{\prime }}}
\end{equation*}%
(where $r^{\prime }=r/(r-1)$), or equivalently, 
\begin{equation}
\left\Vert B(\boldsymbol{u})\right\Vert _{L^{p^{\prime }}(0,T;\mathbb{V}%
^{\prime })}\leq C\left\Vert \boldsymbol{u}\right\Vert _{L^{p}(0,T;\mathbb{V}%
)}^{2\alpha }\left\Vert \boldsymbol{u}\right\Vert _{L^{\infty }(0,T;\mathbb{H%
})}^{2-2\alpha }  \label{0.6'}
\end{equation}%
for any $\boldsymbol{u}\in L^{p}(0,T;\mathbb{V})\cap L^{\infty }(0,T;\mathbb{%
H})$.

We have proven the following inequalities: 
\begin{equation}
\left\{ 
\begin{array}{l}
\text{If }p\geq 3\text{, }\left\Vert B(\boldsymbol{u})\right\Vert
_{L^{p^{\prime }}(0,T;\mathbb{V}^{\prime })}\leq c_{2}\left\Vert \boldsymbol{%
u}\right\Vert _{L^{p}(0,T;\mathbb{V})}^{2}\text{ for all }\boldsymbol{u}\in
L^{p}(0,T;\mathbb{V}); \\ 
\text{If }1+\frac{2N}{N+2}\leq p<3\text{, }\left\Vert B(\boldsymbol{u}%
)\right\Vert _{L^{p^{\prime }}(0,T;\mathbb{V}^{\prime })}\leq C\left\Vert 
\boldsymbol{u}\right\Vert _{L^{p}(0,T;\mathbb{V})}^{2\alpha }\left\Vert 
\boldsymbol{u}\right\Vert _{L^{\infty }(0,T;\mathbb{H})}^{2-2\alpha } \\ 
\text{for all }\boldsymbol{u}\in L^{p}(0,T;\mathbb{V})\cap L^{\infty }(0,T;%
\mathbb{H}).%
\end{array}%
\right.  \label{0.7}
\end{equation}

The above inequalities will be very useful in the sequel. For the sake of
completeness, we choose $\boldsymbol{u}^{0}\in \mathbb{H}$. We are therefore
concerned with the existence of a solution of (\ref{0.1}). The first result
of the work is the following

\begin{theorem}
\label{t1.1}Let $1+\frac{2N}{N+2}\leq p<\infty $. Suppose $\boldsymbol{u}%
^{0}\in \mathbb{H}$. Under assumptions (\textbf{A1})\textbf{-}(\textbf{A4}),
there exists (for each fixed $\varepsilon >0$) a couple $(\boldsymbol{u}%
_{\varepsilon },q_{\varepsilon })\in L^{p}(0,T;\mathbb{V})\cap L^{\infty
}(0,T;\mathbb{H})\times W^{-1,\infty }(0,T;L^{p^{\prime }}(Q))$ solution to 
\emph{(\ref{0.1})}. The function $\boldsymbol{u}_{\varepsilon }$ also
belongs to $\mathcal{C}([0,T];\mathbb{H})$ and $q_{\varepsilon }$ is unique
up to a constant function of $x$: $\int_{Q}q_{\varepsilon }dx=0$.
\end{theorem}

\begin{proof}
Multiplying Eq. (\ref{0.1}) by $\boldsymbol{v}\in \mathbb{V}$ and
integrating over $(0,t)\times Q$ we get\emph{\ }%
\begin{equation*}
\left\{ 
\begin{array}{l}
\left( \boldsymbol{u}_{\varepsilon }(t),\boldsymbol{v}\right) _{\varepsilon
}+\int_{0}^{t}\left\langle \mathcal{A}^{\varepsilon }(s)\boldsymbol{u}%
_{\varepsilon }(s)+B(\boldsymbol{u}_{\varepsilon }(s)),\boldsymbol{v}%
\right\rangle ds=\left( \boldsymbol{u}^{0},\boldsymbol{v}\right)
_{\varepsilon } \\ 
+\int_{0}^{t}\left( f\left( \frac{s}{\varepsilon ^{2}},\boldsymbol{u}%
_{\varepsilon }(s)\right) ,\boldsymbol{v}\right) _{\varepsilon }ds\text{\
for all\emph{\ }}\boldsymbol{v}\in \mathbb{V}\text{\ and a.e.\emph{\ }}0<t<T%
\emph{,}%
\end{array}%
\right.
\end{equation*}%
which, in view of the properties of the operators $\mathcal{A}^{\varepsilon
} $ and $B$ (see especially (\ref{0.7})), amounts to find a function $%
\boldsymbol{u}_{\varepsilon }\in L^{p}(0,T;\mathbb{V})$ such that 
\begin{equation}
\frac{d\boldsymbol{u}_{\varepsilon }}{dt}+\mathcal{A}^{\varepsilon }%
\boldsymbol{u}_{\varepsilon }+B(\boldsymbol{u}_{\varepsilon
})=f^{\varepsilon }(\cdot ,\boldsymbol{u}_{\varepsilon })\text{ in }%
L^{p^{\prime }}(0,T;\mathbb{V}^{\prime })\text{,\ }\boldsymbol{u}%
_{\varepsilon }(0)=\boldsymbol{u}^{0}\text{ in }\mathbb{H}.  \label{0.8}
\end{equation}%
Conversely, a solution of (\ref{0.8}) will satisfy (\ref{0.1}) for a
suitable choice of $q_{\varepsilon }$ which shall be specified later.
Therefore, in view of the properties of the operators $\mathcal{A}%
^{\varepsilon }$ and $B$ (see once again (\ref{0.7})), we can argue as in 
\cite[Theorem 2.1]{Caraballo1} (see also \cite{Caraballo2}) to get the
existence of a solution to (\ref{0.1}) in the space $L^{p}(0,T;\mathbb{V}%
)\cap L^{\infty }(0,T;\mathbb{H})$. It is to be noted that the solution $%
\boldsymbol{u}_{\varepsilon }$ satisfies $\frac{\partial \boldsymbol{u}%
_{\varepsilon }}{\partial t}\in L^{p^{\prime }}(0,T;\mathbb{V}^{\prime })$,
so that, by a well-known result, $\boldsymbol{u}_{\varepsilon }\in \mathcal{C%
}([0,T];\mathbb{H})$.

For the existence of the pressure, we have $\rho ^{\varepsilon
}f^{\varepsilon }(\cdot ,\boldsymbol{u}_{\varepsilon })\in L^{2}(Q_{T})^{N}$%
, so that the necessary condition of \cite[Section 4]{Simon2} for the
existence of the pressure is satisfied. Now, coming back to (\ref{0.1}) and
denoting there 
\begin{equation*}
\boldsymbol{w}_{\varepsilon }=\rho ^{\varepsilon }f^{\varepsilon }(\cdot ,%
\boldsymbol{u}_{\varepsilon })-\rho ^{\varepsilon }\frac{\partial 
\boldsymbol{u}_{\varepsilon }}{\partial t}+\Div\left( a^{\varepsilon }\nabla 
\boldsymbol{u}_{\varepsilon }+b^{\varepsilon }\left\vert \nabla \boldsymbol{u%
}_{\varepsilon }\right\vert ^{p-2}\nabla \boldsymbol{u}_{\varepsilon
}\right) -(\boldsymbol{u}_{\varepsilon }\cdot \nabla )\boldsymbol{u}%
_{\varepsilon }
\end{equation*}%
we have $\left\langle \boldsymbol{w}_{\varepsilon },\boldsymbol{v}%
\right\rangle =0$ for all $\boldsymbol{v}\in \mathcal{V}$ (that is for all $%
\boldsymbol{v}\in \mathcal{C}_{0}^{\infty }(Q)^{N}$ with $\Div\boldsymbol{v}%
=0$) where $\left\langle \cdot ,\cdot \right\rangle $ is the duality pairing
between $\mathcal{D}^{\prime }(Q)^{N}$ and $\mathcal{D}(Q)^{N}$. Next,
arguing as in the proof of \cite[Proposition 5]{Simon2} we are led to $%
\boldsymbol{w}_{\varepsilon }\in W^{-1,\infty }(0,T;W^{-1,p^{\prime
}}(Q)^{N})$, so that there exists a unique $q_{\varepsilon }\in W^{-1,\infty
}(0,T;L^{p^{\prime }}(Q))$ such that 
\begin{equation*}
\nabla q_{\varepsilon }=\boldsymbol{w}_{\varepsilon },\
\int_{Q}q_{\varepsilon }dx=0\text{.}
\end{equation*}%
This completes the proof.
\end{proof}

\subsection{A priori estimates and compactness}

The following result holds.

\begin{lemma}
\label{l2.1}Let $1+\frac{2N}{N+2}\leq p<\infty $. Under assumptions (\textbf{%
A1})\textbf{-}(\textbf{A4}) we have the following estimates: 
\begin{equation}
\sup_{0\leq t\leq T}\left\vert \boldsymbol{u}_{\varepsilon }(t)\right\vert
^{2}\leq C,  \label{2.1}
\end{equation}%
\begin{equation}
\int_{0}^{T}\Vert \boldsymbol{u}_{\varepsilon }(t)\Vert ^{p}dt\leq C
\label{2.2}
\end{equation}

where $C$ is a positive constant which does not depends on $\varepsilon $.
\end{lemma}

\begin{proof}
The variational formulation of (\ref{0.1}) gives, for any $\boldsymbol{v}\in 
\mathbb{V}$, 
\begin{equation}
\frac{d}{dt}\left( \boldsymbol{u}_{\varepsilon }(t),\boldsymbol{v}\right)
_{\varepsilon }+a_{I}^{\varepsilon }(t;\boldsymbol{u}_{\varepsilon }(t),%
\boldsymbol{v})+b_{I}(\boldsymbol{u}_{\varepsilon }(t),\boldsymbol{u}%
_{\varepsilon }(t),\boldsymbol{v})=\left( f\left( \frac{t}{\varepsilon ^{2}},%
\boldsymbol{u}_{\varepsilon }(t)\right) ,\boldsymbol{v}\right)_{\varepsilon
}.  \label{2.3}
\end{equation}%
In all what follows, $C$ is a generic constant that may change from line to
line.

Taking the particular $\boldsymbol{v}=\boldsymbol{u}_{\varepsilon }(t)$ in (%
\ref{2.3}) and using the relation $b_{I}(\boldsymbol{u}_{\varepsilon }(t),%
\boldsymbol{u}_{\varepsilon }(t),\boldsymbol{u}_{\varepsilon }(t))=0$, we
get (after integrating over $[0,t]$) for all $t\in \lbrack 0,T]$, 
\begin{eqnarray*}
&&\left\vert \boldsymbol{u}_{\varepsilon }(t)\right\vert _{\varepsilon
}^{2}+2\nu _{0}\int_{0}^{t}\left\Vert \boldsymbol{u}_{\varepsilon
}(s)\right\Vert _{H_{0}^{1}(Q)^{N}}^{2}ds+2\nu _{1}\int_{0}^{t}\left\Vert 
\boldsymbol{u}_{\varepsilon }(s)\right\Vert ^{p}ds \\
&\leq &\left\vert \boldsymbol{u}^{0}\right\vert _{\varepsilon
}^{2}+2\int_{0}^{t}\left\vert f\left( \frac{s}{\varepsilon ^{2}},\boldsymbol{%
u}_{\varepsilon }(s)\right) \right\vert _{\varepsilon }\left\vert 
\boldsymbol{u}_{\varepsilon }(s)\right\vert _{\varepsilon }ds.
\end{eqnarray*}%
From assumption (\textbf{A4}) we have that 
\begin{equation*}
\left\vert f\left( \frac{s}{\varepsilon ^{2}},\boldsymbol{u}_{\varepsilon
}(s)\right) \right\vert _{\varepsilon }\leq C(1+\left\vert \boldsymbol{u}%
_{\varepsilon }(s)\right\vert_{\varepsilon }),
\end{equation*}%
hence 
\begin{eqnarray*}
2\int_{0}^{t}\left\vert f\left( \frac{s}{\varepsilon ^{2}},\boldsymbol{u}%
_{\varepsilon }(s)\right) \right\vert _{\varepsilon }\left\vert \boldsymbol{u%
}_{\varepsilon }(s)\right\vert _{\varepsilon }ds &\leq
&C\int_{0}^{t}(1+\left\vert \boldsymbol{u}_{\varepsilon
}(s)\right\vert_{\varepsilon })\left\vert \boldsymbol{u}_{\varepsilon
}(s)\right\vert _{\varepsilon }ds \\
&\leq &C+C\int_{0}^{t}\left\vert \boldsymbol{u}_{\varepsilon }(s)\right\vert
_{\varepsilon }^{2}ds.
\end{eqnarray*}%
We therefore get 
\begin{equation}
\begin{array}{l}
\left\vert \boldsymbol{u}_{\varepsilon }(t)\right\vert _{\varepsilon
}^{2}+2\nu _{0}\int_{0}^{t}\left\Vert \boldsymbol{u}_{\varepsilon
}(s)\right\Vert _{H_{0}^{1}(Q)^{N}}^{2}ds+2\nu _{1}\int_{0}^{t}\left\Vert 
\boldsymbol{u}_{\varepsilon }(s)\right\Vert ^{p}ds \\ 
\ \ \ \leq C+C\int_{0}^{t}\left\vert \boldsymbol{u}_{\varepsilon
}(s)\right\vert _{\varepsilon }^{2}ds%
\end{array}
\label{2.4}
\end{equation}%
since $\sup_{\varepsilon >0}\left\vert \boldsymbol{u}^{0}\right\vert
_{\varepsilon }^{2}<\infty $. We readily deduce from (\ref{2.4}) 
\begin{equation*}
\left\vert \boldsymbol{u}_{\varepsilon }(t)\right\vert _{\varepsilon
}^{2}\leq C+C\int_{0}^{t}\left\vert \boldsymbol{u}_{\varepsilon
}(s)\right\vert _{\varepsilon }^{2}ds.
\end{equation*}%
By the application of Gronwall's inequality we find that 
\begin{equation}
\left\vert \boldsymbol{u}_{\varepsilon }(t)\right\vert _{\varepsilon
}^{2}\leq C\text{ for }0\leq t\leq T,\ \varepsilon >0  \label{2.6}
\end{equation}%
where $C$ is a positive constant independent of $\varepsilon $ and $t$. But $%
\left\vert \boldsymbol{u}_{\varepsilon }(t)\right\vert _{\varepsilon
}^{2}\geq \Lambda ^{-1}\left\vert \boldsymbol{u}_{\varepsilon
}(t)\right\vert ^{2}$, so that (\ref{2.1}) comes from (\ref{2.6}). We also
infer from (\ref{2.4}) that (\ref{2.2}) holds true. Still from (\ref{2.4})
it emerges that 
\begin{equation}
\sup_{\varepsilon >0}\int_{0}^{T}\left\Vert \boldsymbol{u}_{\varepsilon
}(t)\right\Vert _{H_{0}^{1}(Q)^{N}}^{2}dt\leq C.  \label{2.8}
\end{equation}
\end{proof}

The next result will be of great interest in the homogenization process.

\begin{proposition}
\label{p2.2}The sequence $(\boldsymbol{u}_{\varepsilon })_{\varepsilon >0}$
is relatively compact in the space $L^{2}(Q_{T})^{N}$.
\end{proposition}

\begin{proof}
Let 
\begin{equation*}
\mathbb{W}=\left\{ \boldsymbol{u}\in L^{p}(0,T;\mathbb{V}):\frac{\partial 
\boldsymbol{u}}{\partial t}\in L^{p^{\prime }}(0,T;\mathbb{V}^{\prime
}+L^{2}(Q)^{N})\right\} .
\end{equation*}%
Since $\mathbb{V}\hookrightarrow \mathbb{H}$ is compact and $\mathbb{H}%
\hookrightarrow L^{2}(Q)^{N}$, we have that $\mathbb{V}\hookrightarrow
L^{2}(Q)^{N}$ is compact. We also have $L^{2}(Q)^{N}\hookrightarrow \mathbb{V%
}^{\prime }+L^{2}(Q)^{N}$, hence the compactness of $\mathbb{W}$ in $%
L^{2}(0,T;L^{2}(Q)^{N})=L^{2}(Q_{T})^{N}$ by a well-known result; see e.g. 
\cite[p. 58, Theorem 5.1]{Lions}. With this in mind we need to check that
the sequence $(\boldsymbol{u}_{\varepsilon })_{\varepsilon >0}$ is bounded
in $\mathbb{W}$. But since $\mathbb{V}^{\prime }\hookrightarrow \mathbb{V}%
^{\prime }+L^{2}(Q)^{N}$ we just need to verify (after (\ref{2.2})) that $%
(\partial \boldsymbol{u}_{\varepsilon }/\partial t)_{\varepsilon >0}$ is
bounded in $L^{p^{\prime }}(0,T;\mathbb{V}^{\prime })$. But this comes from
the combination of (\ref{2.1})-(\ref{2.2}) with (\ref{0.7}) and (\ref{2.8}).
\end{proof}

We can now deal with the estimation of the pressure. Before we can do that,
let us recall the definition and properties of the Bogovskii operator. Let 
\begin{equation*}
L_{0}^{p}(Q)=\left\{ v\in L^{p}(Q):\int_{Q}vdx=0\right\}
\end{equation*}
for $1<p<\infty $. We have the following result.

\begin{lemma}[{\protect\cite[Lemma 3.17, p.169]{Novotny}}]
\label{l2.0}Let $1<p<\infty $. There exists a linear operator $\mathcal{B}%
:L_{0}^{p}(Q)\rightarrow W_{0}^{1,p}(Q)^{N}$ with the following properties:

\begin{itemize}
\item[(i)] $\Div\mathcal{B}(f)=f$ a.e. in $Q$ for any $f\in L_{0}^{p}(Q)$

\item[(ii)] $\left\Vert \mathcal{B}(f)\right\Vert _{W_{0}^{1,p}(Q)^{N}}\leq
c(p,Q)\left\Vert f\right\Vert _{L^{p}(Q)}$

\item[(iii)] If $f=\Div g$ with $g\in L^{r}(Q)^{N}$ and $g\cdot n=0$ on $%
\partial Q$ for some $1<r<\infty $ where $n$ is an outward unit vector
normal to $\partial Q$, then 
\begin{equation*}
\left\Vert \mathcal{B}(f)\right\Vert _{L^{r}(Q)^{N}}\leq c(r,Q)\left\Vert
g\right\Vert _{L^{r}(Q)^{N}}.
\end{equation*}

\item[(iv)] If $f\in \mathcal{C}_{0}^{\infty }(Q)\cap L_{0}^{p}(Q)$ then $%
\mathcal{B}(f)\in \mathcal{C}_{0}^{\infty }(Q)^{N}$.
\end{itemize}
\end{lemma}

With this in mind, the following result holds.

\begin{lemma}
\label{l2.0'}Let $1+\frac{2N}{N+2}\leq p<\infty $ and $p^{\prime }$ its
conjugate. We have $q_{\varepsilon }\in L^{p^{\prime }}(0,T;L_{0}^{p^{\prime
}}(Q))$ with 
\begin{equation}
\sup_{\varepsilon >0}\left\Vert q_{\varepsilon }\right\Vert _{L^{p^{\prime
}}(Q_{T})}\leq C  \label{2.15}
\end{equation}%
for some positive constant $C$ independent of $\varepsilon $.
\end{lemma}

\begin{proof}
First we know that $q_{\varepsilon }\in \mathcal{D}^{\prime }(Q_{T})$. Let $%
g\in \mathcal{C}_{0}^{\infty }(Q_{T})\cap L^{p}(0,T;L_{0}^{p}(Q))$ (recall
that the dual of $L_{0}^{p^{\prime }}(Q)$ is $L_{0}^{p}(Q)$). In view of
[part (iv) of] Lemma \ref{l2.0} let $\boldsymbol{v}\in \mathcal{C}%
_{0}^{\infty }(Q_{T})^{N}$ be such that $\Div\boldsymbol{v}=g$. We have 
\begin{equation*}
\left\langle \nabla q_{\varepsilon },\boldsymbol{v}\right\rangle
=-\left\langle q_{\varepsilon },\Div\boldsymbol{v}\right\rangle
=-\left\langle q_{\varepsilon },g\right\rangle ,
\end{equation*}%
that is (because of (\ref{0.1})) 
\begin{equation*}
\left\langle q_{\varepsilon },g\right\rangle =-\left\langle \nabla
q_{\varepsilon },\boldsymbol{v}\right\rangle =\left\langle -\boldsymbol{w}%
_{\varepsilon },\boldsymbol{v}\right\rangle
\end{equation*}%
where 
\begin{equation*}
\boldsymbol{w}_{\varepsilon }=\rho ^{\varepsilon }f^{\varepsilon }(\cdot ,%
\boldsymbol{u}_{\varepsilon })-\rho ^{\varepsilon }\frac{\partial 
\boldsymbol{u}_{\varepsilon }}{\partial t}+\Div\left( a^{\varepsilon }\nabla 
\boldsymbol{u}_{\varepsilon }+b^{\varepsilon }\left\vert \nabla \boldsymbol{u%
}_{\varepsilon }\right\vert ^{p-2}\nabla \boldsymbol{u}_{\varepsilon
}\right) -(\boldsymbol{u}_{\varepsilon }\cdot \nabla )\boldsymbol{u}%
_{\varepsilon }.
\end{equation*}%
But we infer from Lemma \ref{l2.1} that 
\begin{eqnarray*}
\left\vert \left\langle q_{\varepsilon },g\right\rangle \right\vert &\leq
&C\left( \left\Vert v\right\Vert _{L^{p}(0,T;W_{0}^{1,p}(Q)^{N})}+\left\Vert
v\right\Vert _{L^{p}(Q_{T})^{N}}\right) \\
&\leq &C\left\Vert \boldsymbol{v}\right\Vert _{L^{p}(0,T;W_{0}^{1,p}(Q)^{N})}
\\
&\leq &C\left\Vert g\right\Vert _{L^{p}(Q_{T})},
\end{eqnarray*}%
the last inequality above being due to [part (ii) of] Lemma \ref{l2.0}. We
therefore deduce from the above inequality that $q_{\varepsilon }\in
L^{p^{\prime }}(Q_{T})$, that is $q_{\varepsilon }\in L^{p^{\prime
}}(0,T;L_{0}^{p^{\prime }}(Q))$ with 
\begin{equation*}
\left\Vert q_{\varepsilon }\right\Vert _{L^{p^{\prime }}(Q_{T})}\leq C
\end{equation*}%
for a constant $C>0$ independent of $\varepsilon $.
\end{proof}

\section{The $\Sigma $-convergence method}

We begin this section by collecting some useful tools about almost
periodicity.

\subsection{Almost periodic functions}

The concept of almost periodic functions is well known in the literature.
Following \cite[Section 2]{RSW}, we present in this section some basic facts
about it, which will be used throughout the paper. For a general
presentation and an efficient treatment of this concept, we refer to \cite%
{Bohr}, \cite{Besicovitch} and \cite{Levitan}.

Let $\mathcal{B}(\mathbb{R}^{N})$ denote the Banach algebra of bounded
continuous (complex-valued) functions on $\mathbb{R}^{N}$ endowed with the $%
\sup $ norm topology.

A mapping $u:\mathbb{R}^{N}\rightarrow \mathbb{C}$ is called an almost
periodic function, or a Bohr almost periodic function, if $u\in \mathcal{B}(%
\mathbb{R}^{N})$ and further the set of all its translates $\{u(\cdot
+a):a\in \mathbb{R}^{N}\}$ has a compact closure in $\mathcal{B}(\mathbb{R}%
^{N})$. We denote by $AP(\mathbb{R}^{N})$ the set of all continuous almost
periodic functions on $\mathbb{R}^{N}$. $AP(\mathbb{R}^{N})$ is a
commutative $\mathcal{C}^{\ast }$-algebra with identity. Next, let us denote
by Trig$(\mathbb{R}^{N})$ the algebra of all trigonometric polynomials, i.e.
all finite sums of the form%
\begin{equation*}
u(y)=\sum a_{k}\exp (i\xi _{k}\cdot y)\text{, }\xi _{k}\in \mathbb{R}^{N}%
\text{, }a_{k}\in \mathbb{C}.
\end{equation*}%
Then Trig$(\mathbb{R}^{N})$ $\subset AP(\mathbb{R}^{N})$ and further $AP(%
\mathbb{R}^{N})$ coincides with the closure of Trig$(\mathbb{R}^{N})$ in $%
\mathcal{B}(\mathbb{R}^{N})$. From the above definition, one easily sees
that every element of $AP(\mathbb{R}^{N})$ is uniformly continuous. Moreover
it is classically known that $AP(\mathbb{R}^{N})$ enjoys the following
properties:

\begin{itemize}
\item[(P)$_{1}$] $u(\cdot +a)\in AP(\mathbb{R}^{N})$ whenever $u\in AP(%
\mathbb{R}^{N})$ and for every $a\in \mathbb{R}^{N}$;

\item[(P)$_{2}$] For each $u\in AP(\mathbb{R}^{N})$ the closed convex hull
of $\{u(\cdot +a)\}_{a\in \mathbb{R}^{N}}$ in $\mathcal{B}(\mathbb{R}^{N})$
contains a unique complex constant $\mathfrak{M}(u)$ called the mean value
of $u$, and which satisfies the property that the sequence $(u^{\varepsilon
})_{\varepsilon >0}$ (where $u^{\varepsilon }(x)=u(x/\varepsilon )$, $x\in 
\mathbb{R}^{N}$) weakly $\ast $-converges in $L^{\infty }(\mathbb{R}^{N})$
to $\mathfrak{M}(u)$ as $\varepsilon \rightarrow 0$.
\end{itemize}

It follows from the above properties that $AP(\mathbb{R}^{N})$ is an \textit{%
algebra with mean value} on $\mathbb{R}^{N}$ \cite{Jikov}. Its spectrum is
the Bohr compactification of $\mathbb{R}^{N}$, sometimes denoted by $b%
\mathbb{R}^{N}$ in the literature and, in order to simplify the notation, we
denote it here just by $\mathcal{K}$. The set $\mathcal{K}$ is a compact
topological\ Abelian group. The Haar measure on $\mathcal{K}$ is denoted by $%
\beta $. In view of the Gelfand representation theory of $\mathcal{C}^{\ast
} $-algebras we have the next result.

\begin{theorem}
\label{t2.1}There exists an isometric $\ast $-isomorphism $\mathcal{G}$ of $%
AP(\mathbb{R}^{N})$ onto $\mathcal{C}(\mathcal{K})$ such that every element
of $AP(\mathbb{R}^{N})$ is viewed as a restriction to $\mathbb{R}^{N}$ of a
unique element in $\mathcal{C}(\mathcal{K})$. Moreover the mean value $%
\mathfrak{M}$ defined on $AP(\mathbb{R}^{N})$ has an integral representation
in terms of the Haar measure $\beta $ as follows: 
\begin{equation*}
\mathfrak{M}(u)=\int_{\mathcal{K}}\mathcal{G}(u)d\beta \text{\ \ for all }%
u\in AP(\mathbb{R}^{N})\text{.}
\end{equation*}
\end{theorem}

The isometric $\ast $-isomorphism $\mathcal{G}$ of the above theorem is
referred to as the Gelfand transformation. The image $\mathcal{G}(u)$ of $u$
will very often be denoted by $\widehat{u}$.

We introduce the space $AP^{\infty }(\mathbb{R}^{N})=\{u\in AP(\mathbb{R}%
^{N}):D_{y}^{\alpha }u\in AP(\mathbb{R}^{N})$ for every $\alpha =(\alpha
_{1},\ldots ,\alpha _{N})\in \mathbb{N}^{N}\}$ where $D_{y}^{\alpha }=\frac{%
\partial ^{\left\vert \alpha \right\vert }}{\partial y_{1}^{\alpha
_{1}}\ldots \partial y_{N}^{\alpha _{N}}}$. For $m\in \mathbb{N}$ (the non
negative integers) and for $u\in AP^{\infty }(\mathbb{R}^{N})$ we set $%
\left\Vert \left\vert u\right\vert \right\Vert _{m}=\sup_{\left\vert \alpha
\right\vert \leq m}\sup_{y\in \mathbb{R}^{N}}\left\vert D_{y}^{\alpha
}u\right\vert $ (which defines a norm on $AP^{\infty }(\mathbb{R}^{N})$).
Then $AP^{\infty }(\mathbb{R}^{N})$ is a Fr\'{e}chet space with respect to
the natural topology of projective limit, defined by the increasing family
of norms $\left\Vert \left\vert \cdot \right\vert \right\Vert _{m}$ ($m\in 
\mathbb{N}$).

Next, let $B_{AP}^{p}(\mathbb{R}^{N})$ ($1\leq p<\infty $) denote the space
of Besicovitch almost periodic functions on $\mathbb{R}^{N}$, that is the
closure of $AP(\mathbb{R}^{N})$ with respect to the Besicovitch seminorm 
\begin{equation*}
\left\Vert u\right\Vert _{p}=\left( \underset{r\rightarrow +\infty }{\lim
\sup }\frac{1}{\left\vert B_{r}\right\vert }\int_{B_{r}}\left\vert
u(y)\right\vert ^{p}dy\right) ^{1/p}
\end{equation*}%
where $B_{r}$ is the open ball of $\mathbb{R}^{N}$ of radius $r$ centered at
the origin. It is known that $B_{AP}^{p}(\mathbb{R}^{N})$ is a complete
seminormed vector space verifying $B_{AP}^{q}(\mathbb{R}^{N})\subset
B_{AP}^{p}(\mathbb{R}^{N})$ for $1\leq p\leq q<\infty $. Using this last
property one may naturally define the space $B_{AP}^{\infty }(\mathbb{R}%
^{N}) $ as follows: 
\begin{equation*}
B_{AP}^{\infty }(\mathbb{R}^{N})=\{f\in \cap _{1\leq p<\infty }B_{AP}^{p}(%
\mathbb{R}^{N}):\sup_{1\leq p<\infty }\left\Vert f\right\Vert _{p}<\infty \}%
\text{.}\;\;\;\;\;\;\;\;\;
\end{equation*}%
We endow $B_{AP}^{\infty }(\mathbb{R}^{N})$ with the seminorm $\left[ f%
\right] _{\infty }=\sup_{1\leq p<\infty }\left\Vert f\right\Vert _{p}$,
which makes it a complete seminormed space. We recall that the spaces $%
B_{AP}^{p}(\mathbb{R}^{N})$ ($1\leq p\leq \infty $) are not Fr\'{e}chet
spaces since they are not separated. The following properties are worth
noticing \cite{CMP, NA}:

\begin{itemize}
\item[(\textbf{1)}] The Gelfand transformation $\mathcal{G}:AP(\mathbb{R}%
^{N})\rightarrow \mathcal{C}(\mathcal{K})$ extends by continuity to a unique
continuous linear mapping, still denoted by $\mathcal{G}$, of $B_{AP}^{p}(%
\mathbb{R}^{N})$ into $L^{p}(\mathcal{K})$, which in turn induces an
isometric isomorphism $\mathcal{G}_{1}$, of $B_{AP}^{p}(\mathbb{R}^{N})/%
\mathcal{N}=\mathcal{B}_{AP}^{p}(\mathbb{R}^{N})$ onto $L^{p}(\mathcal{K})$
(where $\mathcal{N}=\{u\in B_{AP}^{p}(\mathbb{R}^{N}):\mathcal{G}(u)=0\}$).
Moreover if $u\in B_{AP}^{p}(\mathbb{R}^{N})\cap L^{\infty }(\mathbb{R}^{N})$
then $\mathcal{G}(u)\in L^{\infty }(\mathcal{K})$ and $\left\Vert \mathcal{G}%
(u)\right\Vert _{L^{\infty }(\mathcal{K})}\leq \left\Vert u\right\Vert
_{L^{\infty }(\mathbb{R}^{N})}$.

\item[(\textbf{2)}] The mean value $\mathfrak{M}$, defined on $AP(\mathbb{R}%
^{N})$, extends by continuity to a positive continuous linear form (still
denoted by $\mathfrak{M}$) on $B_{AP}^{p}(\mathbb{R}^{N})$ satisfying $%
\mathfrak{M}(u)=\int_{\mathcal{K}}\mathcal{G}(u)d\beta $ and $\mathfrak{M}%
(u(\cdot +a))=\mathfrak{M}(u)$ for each $u\in B_{AP}^{p}(\mathbb{R}^{N})$
and all $a\in \mathbb{R}^{N}$, where $u(\cdot +a)(y)=u(y+a)$ for almost all $%
y\in \mathbb{R}^{N}$. Moreover for $u\in B_{AP}^{p}(\mathbb{R}^{N})$ we have 
$\left\Vert u\right\Vert _{p}=\left[ \mathfrak{M}(\left\vert u\right\vert
^{p})\right] ^{1/p}$.
\end{itemize}

Spaces of almost periodic functions with values in a Banach space are
defined in a natural way, we refer to \cite{Blot} for details. Keep the
following notations in mind: $AP(\mathbb{R}^{N};\mathbb{C})=AP(\mathbb{R}%
^{N}) $ and $B_{AP}^{p}(\mathbb{R}^{N};\mathbb{C})=B_{AP}^{p}(\mathbb{R}%
^{N}) $.

We can also define the notion of almost periodic distributions. To do this,
let $\mathcal{B}_{\infty }(\mathbb{R}^{N})$ denote the space of all $%
\mathcal{C}^{\infty }$ functions in $\mathbb{R}^{N}$ that are bounded
together with all their derivatives of any order. We equip $\mathcal{B}%
_{\infty }(\mathbb{R}^{N})$ with the locally convex topology defined by the
increasingly filtered separating family of norms $\left\Vert \left\vert
\cdot \right\vert \right\Vert _{m}$ ($m\in \mathbb{N}$). As usual, let $%
\mathcal{D}(\mathbb{R}^{N})$ denote the subspace of $\mathcal{B}_{\infty }(%
\mathbb{R}^{N})$ consisting of functions with compact support. We denote by $%
\mathcal{B}_{\infty }^{\prime }(\mathbb{R}^{N})$ the topological dual of $%
\mathcal{B}_{\infty }(\mathbb{R}^{N})$. We recall that $\mathcal{B}_{\infty
}^{\prime }(\mathbb{R}^{N})$ is not a subspace of $\mathcal{D}^{\prime }(%
\mathbb{R}^{N})$ (the usual space of distributions on $\mathbb{R}^{N}$)
since $\mathcal{D}(\mathbb{R}^{N})$ is not dense in $\mathcal{B}_{\infty }(%
\mathbb{R}^{N})$; see \cite{LS}. We endow $\mathcal{B}_{\infty }^{\prime }(%
\mathbb{R}^{N})$ with the strong dual topology. For a distribution $T\in 
\mathcal{D}^{\prime }(\mathbb{R}^{N})$ we define its translate $\tau _{a}T$ (%
$a\in \mathbb{R}^{N}$) as follows: 
\begin{equation*}
\left\langle \tau _{a}T,\varphi \right\rangle =\left\langle T,\varphi (\cdot
-a)\right\rangle \text{ for any }\varphi \in \mathcal{D}(\mathbb{R}^{N}).
\end{equation*}

With this in mind, we say that a distribution $T$ is an almost periodic
distribution if $T\in \mathcal{B}_{\infty }^{\prime }(\mathbb{R}^{N})$ and
further the set of translates $\{\tau _{a}T:a\in \mathbb{R}^{N}\}$ is
relatively compact in $\mathcal{B}_{\infty }^{\prime }(\mathbb{R}^{N})$. We
denote by $\mathcal{B}_{AP}^{\prime }(\mathbb{R}^{N})$ the space of all
almost periodic distributions on $\mathbb{R}^{N}$. $\mathcal{B}_{AP}^{\prime
}(\mathbb{R}^{N})$ is endowed with the relative topology on $\mathcal{B}%
_{\infty }^{\prime }(\mathbb{R}^{N})$. The following result holds \cite[p.
206, Section 9]{LS}.

\begin{proposition}
\label{p0.1} The following assertions are equivalent.

\begin{itemize}
\item[(i)] $T\in \mathcal{B}_{AP}^{\prime }(\mathbb{R}^{N})$;

\item[(ii)] $T\in \mathcal{B}_{\infty }^{\prime }(\mathbb{R}^{N})$ is a
finite sum of derivatives of functions in $AP(\mathbb{R}^{N})$;

\item[(iii)] $T\in \mathcal{B}_{\infty }^{\prime }(\mathbb{R}^{N})$ and $%
T\ast \varphi \in AP^{\infty }(\mathbb{R}^{N})$ for any $\varphi \in 
\mathcal{D}(\mathbb{R}^{N})$.
\end{itemize}
\end{proposition}

It can be easily checked that $\varphi T\in \mathcal{B}_{AP}^{\prime }(%
\mathbb{R}^{N})$ whenever $\varphi \in AP^{\infty }(\mathbb{R}^{N})$ and $%
T\in \mathcal{B}_{AP}^{\prime }(\mathbb{R}^{N})$. It follows from part (iii)
of the above proposition that $AP^{\infty }(\mathbb{R}^{N})$ is dense in $%
\mathcal{B}_{AP}^{\prime }(\mathbb{R}^{N})$. Since $B_{AP}^{1}(\mathbb{R}%
^{N})$ is the closure of $AP^{\infty }(\mathbb{R}^{N})$ in the Besicovitch
seminorm whereas $\mathcal{B}_{AP}^{\prime }(\mathbb{R}^{N})$ is the closure
of $AP^{\infty }(\mathbb{R}^{N})$ in $\mathcal{B}_{\infty }^{\prime }(%
\mathbb{R}^{N})$ it follows that $B_{AP}^{1}(\mathbb{R}^{N})$ is
continuously embedded in $\mathcal{B}_{AP}^{\prime }(\mathbb{R}^{N})$. As a
first consequence of this, we have the

\begin{lemma}
\label{l0.1}The mean value $\mathfrak{M}$ considered as defined on $%
AP^{\infty }(\mathbb{R}^{N})$ extends by continuity to a unique continuous
linear form still denoted by $\mathfrak{M}$, on $\mathcal{B}_{AP}^{\prime }(%
\mathbb{R}^{N})$.
\end{lemma}

\begin{proof}
The mapping $\mathfrak{M}$ being continuous on $AP^{\infty }(\mathbb{R}^{N})$
(endowed this time with the relative topology on $\mathcal{B}_{\infty
}^{\prime }(\mathbb{R}^{N})$), the result is therefore an immediate
consequence of the density of $AP^{\infty }(\mathbb{R}^{N})$ in $\mathcal{B}%
_{AP}^{\prime }(\mathbb{R}^{N})$.
\end{proof}

$\mathfrak{M}(T)$ is still called the mean value of $T\in \mathcal{B}%
_{AP}^{\prime }(\mathbb{R}^{N})$. Now, let $T\in \mathcal{B}_{AP}^{\prime }(%
\mathbb{R}^{N})$. Using the preceding lemma and the density of $AP^{\infty }(%
\mathbb{R}^{N})$ in $\mathcal{B}_{AP}^{\prime }(\mathbb{R}^{N})$, we may
define the spectrum of $T$ as Sp$(T)=\{k\in \mathbb{R}^{N}:\mathfrak{M}(%
\overline{\gamma }_{k}T)\neq 0\}$, a countable subset of $\mathbb{R}^{N}$,
where $\gamma _{k}(y)=\exp (ik\cdot y)$ for $y\in \mathbb{R}^{N}$ and $%
\overline{\gamma }_{k}$ stands for the complex conjugation of $\gamma _{k}$.
We know from \cite[p.208]{LS} that if $\mathfrak{M}(\overline{\gamma }%
_{k}T)=0$ for any $k\in $Sp$(T)$ then $T=0$.

Let $\mathbb{R}_{y,\tau }^{N+1}=\mathbb{R}_{y}^{N}\times \mathbb{R}_{\tau }$
denote the space $\mathbb{R}^{N}\times \mathbb{R}$ with generic variables $%
(y,\tau )$. It is known that $AP(\mathbb{R}_{y,\tau }^{N+1})=AP(\mathbb{R}%
_{\tau };AP(\mathbb{R}_{y}^{N}))$ is the closure in $\mathcal{B}(\mathbb{R}%
_{y,\tau }^{N+1})$ of the tensor product $AP(\mathbb{R}_{y}^{N})\otimes AP(%
\mathbb{R}_{\tau })$ \cite{Chou}. We may sometimes set $A_{y}=AP(\mathbb{R}%
_{y}^{N})$, $A_{\tau }=AP(\mathbb{R}_{\tau })$ and $A=AP(\mathbb{R}_{y,\tau
}^{N+1})$. Correspondingly, we will denote the mean value on $A_{\zeta }$ ($%
\zeta =y,\tau $) by $\mathfrak{M}_{\zeta }$.

Now let $1\leq p<\infty $ and consider the $N$-parameter group of isometries 
$\{T(y):y\in \mathbb{R}^{N}\}$ defined by 
\begin{equation*}
T(y):\mathcal{B}_{AP}^{p}(\mathbb{R}^{N})\rightarrow \mathcal{B}_{AP}^{p}(%
\mathbb{R}^{N})\text{,\ }T(y)(u+\mathcal{N})=u(\cdot +y)+\mathcal{N}\text{
for }u\in B_{AP}^{p}(\mathbb{R}^{N}).
\end{equation*}%
Since $AP(\mathbb{R}^{N})$ consists of uniformly continuous functions, $%
\{T(y):y\in \mathbb{R}^{N}\}$ is a strongly continuous group in the
following sense: $T(y)(u+\mathcal{N})\rightarrow u+\mathcal{N}$ in $\mathcal{%
B}_{AP}^{p}(\mathbb{R}^{N})$ as $\left\vert y\right\vert \rightarrow 0$.
Using the isometric isomorphism $\mathcal{G}_{1}$ we associated to $%
\{T(y):y\in \mathbb{R}^{N}\}$ the following $N$-parameter group $\{\overline{%
T}(y):y\in \mathbb{R}^{N}\}$ defined by 
\begin{equation*}
\begin{array}{l}
\overline{T}(y):L^{p}(\mathcal{K})\rightarrow L^{p}(\mathcal{K}) \\ 
\overline{T}(y)\mathcal{G}_{1}(u+\mathcal{N})=\mathcal{G}_{1}(T(y)(u+%
\mathcal{N}))=\mathcal{G}_{1}(u(\cdot +y)+\mathcal{N})\text{\ for }u\in
B_{AP}^{p}(\mathbb{R}^{N})\text{.}%
\end{array}%
\end{equation*}%
The group $\{\overline{T}(y):y\in \mathbb{R}^{N}\}$ is also strongly
continuous. The infinitesimal generator of $T(y)$ (resp. $\overline{T}(y)$)
along the $i$th coordinate direction, denoted by $D_{i,p}$ (resp. $\partial
_{i,p}$) is defined by 
\begin{equation*}
\begin{array}{l}
D_{i,p}u=\lim_{s\rightarrow 0}s^{-1}\left( T(se_{i})u-u\right) \text{\ in }%
\mathcal{B}_{AP}^{p}(\mathbb{R}^{N})\text{ } \\ 
\text{(resp. }\partial _{i,p}v=\lim_{s\rightarrow 0}s^{-1}\left( \overline{T}%
(se_{i})v-v\right) \text{\ in }L^{p}(\mathcal{K})\text{),}%
\end{array}%
\end{equation*}%
where we have used the same letter $u$ to denote the equivalence class of an
element $u\in B_{AP}^{p}(\mathbb{R}^{N})$ in $\mathcal{B}_{AP}^{p}(\mathbb{R}%
^{N})$, $e_{i}=(\delta _{ij})_{1\leq j\leq N}$ ($\delta _{ij}$ being the
Kronecker $\delta $). The domain of $D_{i,p}$ (resp. $\partial _{i,p}$) in $%
\mathcal{B}_{AP}^{p}(\mathbb{R}^{N})$ (resp. $L^{p}(\mathcal{K})$) is
denoted by $\mathcal{D}_{i,p}$ (resp. $\mathcal{W}_{i,p}$). It emerges from
the general theory of semigroups \cite[Chap. VIII, Section 1]{DS} that the
following result holds.

\begin{proposition}
\label{p2.1}$\mathcal{D}_{i,p}$ (resp. $\mathcal{W}_{i,p}$) is a vector
subspace of $\mathcal{B}_{AP}^{p}(\mathbb{R}^{N})$ (resp. $L^{p}(\mathcal{K}%
) $); $D_{i,p}:\mathcal{D}_{i,p}\rightarrow \mathcal{B}_{AP}^{p}(\mathbb{R}%
^{N})$ (resp. $\partial _{i,p}:\mathcal{W}_{i,p}\rightarrow L^{p}(\mathcal{K}%
)$) is a linear operator; $\mathcal{D}_{i,p}$ (resp. $\mathcal{W}_{i,p}$) is
dense in $\mathcal{B}_{AP}^{p}(\mathbb{R}^{N})$ (resp. $L^{p}(\mathcal{K})$%
), and the graph of $D_{i,p}$ (resp. $\partial _{i,p}$) is closed in $%
\mathcal{B}_{AP}^{p}(\mathbb{R}^{N})\times \mathcal{B}_{AP}^{p}(\mathbb{R}%
^{N})$ (resp. $L^{p}(\mathcal{K})\times L^{p}(\mathcal{K})$).
\end{proposition}

In the sequel we denote by $\varrho $ the canonical mapping of $B_{AP}^{p}(%
\mathbb{R}^{N})$ onto $\mathcal{B}_{AP}^{p}(\mathbb{R}^{N})$, that is, $%
\varrho (u)=u+\mathcal{N}$ for $u\in B_{AP}^{p}(\mathbb{R}^{N})$. The
following properties are immediate. The verification can be found either in 
\cite[Chap. B1]{Vo-Khac} or in \cite{Blot}.

\begin{lemma}
\label{l3.1}Let $1\leq i\leq N$. \emph{(1)} If $u\in AP^{1}(\mathbb{R}^{N})$
then $\varrho (u)\in \mathcal{D}_{i,p}$ and 
\begin{equation}
D_{i,p}\varrho (u)=\varrho \left( \frac{\partial u}{\partial y_{i}}\right)
.\ \ \ \ \ \ \ \ \ \ \ \ \ \ \ \ \ \ \ \ \ \ \ \ \ \ \ \ \ \ \ \ \ 
\label{2.2''}
\end{equation}%
\emph{(2)} If $u\in \mathcal{D}_{i,p}$ then $\mathcal{G}_{1}(u)\in \mathcal{W%
}_{i,p}$ and $\mathcal{G}_{1}(D_{i,p}u)=\partial _{i,p}\mathcal{G}_{1}(u)$.
\end{lemma}

We define the higher order derivatives as follows: $D_{p}^{\alpha
}=D_{1,p}^{\alpha _{1}}\circ \cdot \cdot \cdot \circ D_{N,p}^{\alpha _{N}}$
(resp. $\partial _{p}^{\alpha }=\partial _{1,p}^{\alpha _{1}}\circ \cdot
\cdot \cdot \circ \partial _{N,p}^{\alpha _{N}}$) for $\alpha =(\alpha
_{1},...,\alpha _{N})\in \mathbb{N}^{N}$ with $D_{i,p}^{\alpha
_{i}}=D_{i,p}\circ \cdot \cdot \cdot \circ D_{i,p}$, $\alpha _{i}$-times.
Now, set 
\begin{equation*}
\mathcal{B}_{AP}^{1,p}(\mathbb{R}^{N})=\cap _{i=1}^{N}\mathcal{D}%
_{i,p}=\{u\in \mathcal{B}_{AP}^{p}(\mathbb{R}^{N}):D_{i,p}u\in \mathcal{B}%
_{AP}^{p}(\mathbb{R}^{N})\ \forall 1\leq i\leq N\}
\end{equation*}%
and 
\begin{equation*}
\mathcal{D}_{AP}(\mathbb{R}^{N})=\{u\in \mathcal{B}_{AP}^{\infty }(\mathbb{R}%
^{N}):D_{\infty }^{\alpha }u\in \mathcal{B}_{AP}^{\infty }(\mathbb{R}^{N})\
\forall \alpha \in \mathbb{N}^{N}\}.
\end{equation*}%
One can show that $\mathcal{D}_{AP}(\mathbb{R}^{N})$ is dense in $\mathcal{B}%
_{AP}^{p}(\mathbb{R}^{N})$, $1\leq p<\infty $. Equipped with the norm 
\begin{equation*}
\left\Vert u\right\Vert _{\mathcal{B}_{AP}^{1,p}(\mathbb{R}^{N})}=\left(
\left\Vert u\right\Vert _{p}^{p}+\sum_{i=1}^{N}\left\Vert
D_{i,p}u\right\Vert _{p}^{p}\right) ^{1/p}\ \ (u\in \mathcal{B}_{AP}^{1,p}(%
\mathbb{R}^{N})),
\end{equation*}%
$\mathcal{B}_{AP}^{1,p}(\mathbb{R}^{N})$ is a Banach space: this comes from
the fact that the graph of $D_{i,p}$ is closed.

The counter-part of the above properties also holds with 
\begin{equation*}
W^{1,p}(\mathcal{K})=\cap _{i=1}^{N}\mathcal{W}_{i,p}\text{\ in place of }%
\mathcal{B}_{AP}^{1,p}(\mathbb{R}^{N})
\end{equation*}%
and 
\begin{equation*}
\mathcal{D}(\mathcal{K})=\{u\in L^{\infty }(\mathcal{K}):\partial _{\infty
}^{\alpha }u\in L^{\infty }(\mathcal{K})\ \forall \alpha \in \mathbb{N}^{N}\}%
\text{\ in that of }\mathcal{D}_{AP}(\mathbb{R}^{N})\text{.}
\end{equation*}%
Moreover the restriction of $\mathcal{G}_{1}$ to $\mathcal{B}_{AP}^{1,p}(%
\mathbb{R}^{N})$ is an isometric isomorphism of $\mathcal{B}_{AP}^{1,p}(%
\mathbb{R}^{N})$ onto $W^{1,p}(\mathcal{K})$; this is a consequence of [Part
(2) of] Lemma \ref{l3.1}.

Let $u\in \mathcal{D}_{i,p}$ ($p\geq 1$, $1\leq i\leq N$). Then the
inequality 
\begin{equation*}
\left\Vert s^{-1}(T(se_{i})u-u)-D_{i,p}u\right\Vert _{1}\leq c\left\Vert
s^{-1}(T(se_{i})u-u)-D_{i,p}u\right\Vert _{p}
\end{equation*}%
for a positive constant $c$ independent of $u$ and $s$ yields $%
D_{i,1}u=D_{i,p}u$, so that $D_{i,p}$ is the restriction to $\mathcal{B}%
_{AP}^{p}(\mathbb{R}^{N})$ of $D_{i,1}$. Thus for all $u\in \mathcal{D}%
_{i,\infty }$ we have $u\in \mathcal{D}_{i,p}$ ($p\geq 1$) and $D_{i,\infty
}u=D_{i,p}u$\ for any $1\leq i\leq N$. The next result will be useful in the
sequel.

\begin{lemma}[{\protect\cite[Lemma 2]{RSW}}]
\label{l3.2}We have $\mathcal{D}_{AP}(\mathbb{R}^{N})=\varrho (AP^{\infty }(%
\mathbb{R}^{N}))$.
\end{lemma}

From now on, we write $\widehat{u}$ either for $\mathcal{G}(u)$ if $u\in
B_{AP}^{p}(\mathbb{R}^{N})$ or for $\mathcal{G}_{1}(u)$ if $u\in \mathcal{B}%
_{AP}^{p}(\mathbb{R}^{N})$. The following properties are easily verified
(see once again either \cite[Chap. B1]{Vo-Khac} or \cite{Blot}).

\begin{proposition}
\label{p3.2}The following assertions hold.

\begin{itemize}
\item[(i)] $\int_{\mathcal{K}}\partial _{\infty }^{\alpha }\widehat{u}d\beta
=0$ for all $u\in \mathcal{D}_{AP}(\mathbb{R}^{N})$ and $\alpha \in \mathbb{N%
}^{N}$;

\item[(ii)] $\int_{\mathcal{K}}\partial _{i,p}\widehat{u}d\beta =0$ for all $%
u\in \mathcal{D}_{i,p}$ and $1\leq i\leq N$;

\item[(iii)] $D_{i,p}(u\phi )=uD_{i,\infty }\phi +\phi D_{i,p}u$ for all $%
(\phi ,u)\in \mathcal{D}_{AP}(\mathbb{R}^{N})\times \mathcal{D}_{i,p}$ and $%
1\leq i\leq N$.
\end{itemize}
\end{proposition}

From (iii) in Proposition \ref{p3.2} we have 
\begin{equation*}
\int_{\mathcal{K}}\widehat{\phi }\partial _{i,p}\widehat{u}d\beta =-\int_{%
\mathcal{K}}\widehat{u}\partial _{i,\infty }\widehat{\phi }d\beta \ \ \text{%
for all }(u,\phi )\in \mathcal{D}_{i,p}\times \mathcal{D}_{AP}(\mathbb{R}%
^{N}).
\end{equation*}%
This suggests us to define the notion of \textit{distributions} on $\mathcal{%
D}_{AP}(\mathbb{R}^{N})$ and of a weak derivative. First, let us endow $%
\mathcal{D}_{AP}(\mathbb{R}^{N})=\varrho (AP^{\infty }(\mathbb{R}^{N}))$
with its natural topology defined by the family of norms $%
N_{n}(u)=\sup_{\left\vert \alpha \right\vert \leq n}\sup_{y\in \mathbb{R}%
^{N}}\left\vert D_{\infty }^{\alpha }u(y)\right\vert $, integers $n\geq 0$.
So topologized, $\mathcal{D}_{AP}(\mathbb{R}^{N})$ is a Fr\'{e}chet space.
We denote by $\mathcal{D}_{AP}^{\prime }(\mathbb{R}^{N})$ the topological
dual of $\mathcal{D}_{AP}(\mathbb{R}^{N})$ and endow it with the strong dual
topology. The elements of $\mathcal{D}_{AP}^{\prime }(\mathbb{R}^{N})$ are
called the distributions on $\mathcal{D}_{AP}(\mathbb{R}^{N})$. There exists
a relationship between the space of almost periodic distributions $\mathcal{B%
}_{AP}^{\prime }(\mathbb{R}^{N})$ and $\mathcal{D}_{AP}^{\prime }(\mathbb{R}%
^{N})$. To be more precise, we have the following result which allows us to
view $\mathcal{B}_{AP}^{\prime }(\mathbb{R}^{N})$ as a proper subspace of $%
\mathcal{D}_{AP}^{\prime }(\mathbb{R}^{N})$.

\begin{proposition}
\label{p0.2}Given any $T\in \mathcal{B}_{AP}^{\prime }(\mathbb{R}^{N})$, the
linear form $\Phi (T):\mathcal{D}_{AP}(\mathbb{R}^{N})\rightarrow \mathbb{C}$
given by 
\begin{equation*}
\left\langle \Phi (T),\varrho (\varphi )\right\rangle =\mathfrak{M}(\varphi
T)\ \ (\varphi \in AP^{\infty }(\mathbb{R}^{N}))
\end{equation*}%
is a distribution on $\mathcal{D}_{AP}(\mathbb{R}^{N})$. The map $\Phi
:T\mapsto \Phi (T)$ is a linear continuous embedding of $\mathcal{B}%
_{AP}^{\prime }(\mathbb{R}^{N})$ into $\mathcal{D}_{AP}^{\prime }(\mathbb{R}%
^{N})$.
\end{proposition}

\begin{proof}
Let $(\varrho (\varphi _{j}))_{j}$ be a sequence of functions in $\mathcal{D}%
_{AP}(\mathbb{R}^{N})$ such that $\varrho (\varphi _{j})\rightarrow 0$ in $%
\mathcal{D}_{AP}(\mathbb{R}^{N})$; then $\varphi _{j}\rightarrow 0$ in $%
AP^{\infty }(\mathbb{R}^{N})$. In fact assuming $\varrho (\varphi
_{j})\rightarrow 0$ in $\mathcal{D}_{AP}(\mathbb{R}^{N})$ leads to $\mathcal{%
G}_{1}(\varrho (D_{y}^{\alpha }\varphi _{j}))\equiv \mathcal{G}%
(D_{y}^{\alpha }\varphi _{j})\rightarrow 0$ in $\mathcal{C}(\mathcal{K})$
for any $\alpha \in \mathbb{N}^{N}$, so that, as $\left\Vert D_{y}^{\alpha
}\varphi _{j}\right\Vert _{\infty }=\left\Vert \mathcal{G}(D_{y}^{\alpha
}\varphi _{j})\right\Vert _{\infty }$ (see Theorem \ref{t2.1}) our claim is
justified. It follows that $\varphi _{j}\rightarrow 0$ in $\mathcal{B}%
_{AP}^{\prime }(\mathbb{R}^{N})$, hence $\varphi _{j}T\rightarrow 0$ in $%
\mathcal{B}_{AP}^{\prime }(\mathbb{R}^{N})$. We deduce from Lemma \ref{l0.1}
that $\mathfrak{M}(\varphi _{j}T)\rightarrow 0$, from which the continuity
of $\Phi (T)$ on $\mathcal{D}_{AP}(\mathbb{R}^{N})$. Moreover, if $\Phi
(T)=0 $ then $\mathfrak{M}(\varphi T)=0$ for all $\varphi \in AP^{\infty }(%
\mathbb{R}^{N})$, hence $\mathfrak{M}(\overline{\gamma }_{k}T)=0$ for any $%
k\in \mathbb{R}^{N}$. This yields $T=0$ (see \cite[p.208]{LS}). $\Phi $ is
then injective. It remains to check that $\Phi $ is continuous. But this
follows at once by the continuity of the mean value $\mathfrak{M}$ on $%
\mathcal{B}_{AP}^{\prime }(\mathbb{R}^{N})$ (see Lemma \ref{l0.1}).
\end{proof}

\begin{remark}
\label{r0.1}\emph{The map }$\Phi $\emph{\ defined above is not surjective
since }$\mathcal{B}_{AP}^{\prime }(\mathbb{R}^{N})$\emph{\ is not the
topological dual of} $AP^{\infty }(\mathbb{R}^{N})$\emph{\ (see \cite[p.207]%
{LS}). Since it is injective and continuous, }$\mathcal{B}_{AP}^{\prime }(%
\mathbb{R}^{N})$\emph{\ can be viewed as a proper subspace of} $\mathcal{D}%
_{AP}^{\prime }(\mathbb{R}^{N})$\emph{.}
\end{remark}

The weak derivative of $u\in \mathcal{D}_{AP}^{\prime }(\mathbb{R}^{N})$ is
defined as follows: for any $\alpha \in \mathbb{N}^{N}$, $D^{\alpha }u$
stands for the distribution defined by the formula 
\begin{equation*}
\left\langle D^{\alpha }u,\phi \right\rangle =(-1)^{\left\vert \alpha
\right\vert }\left\langle u,D_{\infty }^{\alpha }\phi \right\rangle \text{\
for all }\phi \in \mathcal{D}_{AP}(\mathbb{R}^{N}).
\end{equation*}%
The space $\mathcal{D}_{AP}(\mathbb{R}^{N})$ being dense in $\mathcal{B}%
_{AP}^{p}(\mathbb{R}^{N})$ ($1\leq p<\infty $), it follows that $\mathcal{B}%
_{AP}^{p}(\mathbb{R}^{N})\subset \mathcal{D}_{AP}^{\prime }(\mathbb{R}^{N})$
with continuous embedding, so that the weak derivative of any $f\in \mathcal{%
B}_{AP}^{p}(\mathbb{R}^{N})$ is well defined and verifies the following
functional equation: 
\begin{equation*}
\left\langle D^{\alpha }f,\phi \right\rangle =(-1)^{\left\vert \alpha
\right\vert }\int_{\mathcal{K}}\widehat{f}\partial _{\infty }^{\alpha }%
\widehat{\phi }d\beta \text{\ for all }\phi \in \mathcal{D}_{AP}(\mathbb{R}%
^{N}).
\end{equation*}%
As a special case, for $f\in \mathcal{D}_{i,p}$ we have 
\begin{equation*}
-\int_{\mathcal{K}}\widehat{f}\partial _{i,p}\widehat{\phi }d\beta =\int_{%
\mathcal{K}}\widehat{\phi }\partial _{i,p}\widehat{f}d\beta \ \ \forall \phi
\in \mathcal{D}_{AP}(\mathbb{R}^{N}).
\end{equation*}%
Hence we may identify $D_{i,p}f$ with $D^{\alpha _{i}}f$, $\alpha
_{i}=(\delta _{ij})_{1\leq j\leq N}$. Conversely, if $f\in \mathcal{B}%
_{AP}^{p}(\mathbb{R}^{N})$ is such that there exists $f_{i}\in \mathcal{B}%
_{AP}^{p}(\mathbb{R}^{N})$ with $\left\langle D^{\alpha _{i}}f,\phi
\right\rangle =-\int_{\mathcal{K}}\widehat{f}_{i}\widehat{\phi }d\beta $ for
all $\phi \in \mathcal{D}_{AP}(\mathbb{R}^{N})$, then $f\in \mathcal{D}%
_{i,p} $ and $D_{i,p}f=f_{i}$. This allows us to justify the fact that $%
\mathcal{B}_{AP}^{1,p}(\mathbb{R}^{N})$ is a Banach space under the norm $%
\left\Vert \cdot \right\Vert _{\mathcal{B}_{AP}^{1,p}(\mathbb{R}^{N})}$. The
same is true for $W^{1,p}(\mathcal{K})$. Moreover $\mathcal{D}_{AP}(\mathbb{R%
}^{N})$ (resp. $\mathcal{D}(\mathcal{K})$) is a dense subspace of $\mathcal{B%
}_{AP}^{1,p}(\mathbb{R}^{N})$ (resp. $W^{1,p}(\mathcal{K})$).

We now define the appropriate space of correctors. For that, let $\rho \in
B_{AP}^{2}(\mathbb{R}^{N})\cap L^{\infty }(\mathbb{R}^{N})$ be freely fixed
with $\mathfrak{M}(\rho )>0$ ($\mathfrak{M}(\rho )$ the mean value of $\rho $%
). We begin by defining the following space: 
\begin{equation*}
\mathcal{B}_{AP,\rho }^{1,p}(\mathbb{R}^{N})=\{u\in \mathcal{B}_{AP}^{1,p}(%
\mathbb{R}^{N}):\mathfrak{M}(\rho u)=0\}.
\end{equation*}%
We endow it with the seminorm 
\begin{equation*}
\left\Vert u\right\Vert _{\#,p}=\left( \sum_{i=1}^{N}\left\Vert
D_{i,p}u\right\Vert _{p}^{p}\right) ^{1/p}\ \ (u\in \mathcal{B}_{AP,\rho
}^{1,p}(\mathbb{R}^{N})).
\end{equation*}%
Since $\mathfrak{M}(\rho )\neq 0$ it follows that $\left\Vert \cdot
\right\Vert _{\#,p}$ is actually a norm on $\mathcal{B}_{AP,\rho }^{1,p}(%
\mathbb{R}^{N})$. Under this norm $\mathcal{B}_{AP,\rho }^{1,p}(\mathbb{R}%
^{N})$ is unfortunately not complete. We denote by $\mathcal{B}_{\#AP}^{1,p}(%
\mathbb{R}^{N})$ its completion with respect to the above norm and by $J_{p}$
the canonical embedding of $\mathcal{B}_{AP,\rho }^{1,p}(\mathbb{R}^{N})$
into $\mathcal{B}_{\#AP}^{1,p}(\mathbb{R}^{N})$. It can be easily checked
that $\mathcal{D}_{AP,\rho }(\mathbb{R}^{N})=\{u\in \mathcal{D}_{AP}(\mathbb{%
R}^{N}):\mathfrak{M}(\rho u)=0\}$ is dense in $\mathcal{B}_{AP,\rho }^{1,p}(%
\mathbb{R}^{N})$. It holds that:

\begin{itemize}
\item[(P$_{1}$)] The gradient operator $D_{p}=(D_{1,p},...,D_{N,p}):\mathcal{%
B}_{AP,\rho }^{1,p}(\mathbb{R}^{N})\rightarrow (\mathcal{B}_{AP}^{p}(\mathbb{%
R}^{N}))^{N}$ extends by continuity to a unique mapping $\overline{D}_{p}:%
\mathcal{B}_{\#AP}^{1,p}(\mathbb{R}^{N})\rightarrow (\mathcal{B}_{AP}^{p}(%
\mathbb{R}^{N}))^{N}$ with the properties 
\begin{equation*}
D_{i,p}=\overline{D}_{i,p}\circ J_{p}
\end{equation*}%
and 
\begin{equation*}
\left\Vert u\right\Vert _{\#,p}=\left( \sum_{i=1}^{N}\left\Vert \overline{D}%
_{i,p}u\right\Vert _{p}^{p}\right) ^{1/p}\ \ \text{for }u\in \mathcal{B}%
_{\#AP}^{1,p}(\mathbb{R}^{N}).
\end{equation*}

\item[(P$_{2}$)] The space $J_{p}(\mathcal{B}_{AP,\rho }^{1,p}(\mathbb{R}%
^{N}))$ (and hence $J_{p}(\mathcal{D}_{AP,\rho }(\mathbb{R}^{N}))$) is dense
in $\mathcal{B}_{\#AP}^{1,p}(\mathbb{R}^{N})$.
\end{itemize}

\noindent Moreover the mapping $\overline{D}_{p}$ is an isometric embedding
of $\mathcal{B}_{\#AP}^{1,p}(\mathbb{R}^{N})$ onto a closed subspace of $(%
\mathcal{B}_{AP}^{p}(\mathbb{R}^{N}))^{N}$, so that $\mathcal{B}%
_{\#AP}^{1,p}(\mathbb{R}^{N})$ is a reflexive Banach space. By duality we
define the divergence operator div$_{p^{\prime }}:(\mathcal{B}_{AP}^{p}(%
\mathbb{R}^{N}))^{N}\rightarrow (\mathcal{B}_{\#AP}^{1,p}(\mathbb{R}%
^{N}))^{\prime }$ ($p^{\prime }=p/(p-1)$) by 
\begin{equation}
\left\langle \text{div}_{p^{\prime }}u,v\right\rangle =-\left\langle u,%
\overline{D}_{p}v\right\rangle \text{\ for }v\in \mathcal{B}_{\#AP}^{1,p}(%
\mathbb{R}^{N})\text{ and }u=(u_{i})\in (\mathcal{B}_{AP}^{p^{\prime }}(%
\mathbb{R}^{N}))^{N}\text{,}  \label{2.6''}
\end{equation}%
where $\left\langle u,\overline{D}_{p}v\right\rangle =\sum_{i=1}^{N}\int_{%
\mathcal{K}}\widehat{u}_{i}\partial _{i,p}\widehat{v}d\beta $. The just
defined div$_{p^{\prime }}$ operator extends the natural divergence operator
defined in $\mathcal{D}_{AP}(\mathbb{R}^{N})$ since $D_{i,p}f=\overline{D}%
_{i,p}(J_{p}f)$ for all $f\in \mathcal{D}_{AP}(\mathbb{R}^{N})$.

Now taking in (\ref{2.6''}) $u=D_{p^{\prime }}w$ with $w\in \mathcal{B}%
_{AP}^{p^{\prime }}(\mathbb{R}^{N})$ being such that $D_{p^{\prime }}w\in (%
\mathcal{B}_{AP}^{p^{\prime }}(\mathbb{R}^{N}))^{N}$, we define the
Laplacian operator on $\mathcal{B}_{AP}^{p^{\prime }}(\mathbb{R}^{N})$
(denoted here by $\Delta _{p^{\prime }}$) as follows: 
\begin{equation}
\left\langle \Delta _{p^{\prime }}w,v\right\rangle =\left\langle \text{div}%
_{p^{\prime }}(D_{p^{\prime }}w),v\right\rangle =-\left\langle D_{p^{\prime
}}w,\overline{D}_{p}v\right\rangle \text{\ for all }v\in \mathcal{B}%
_{\#AP}^{1,p}(\mathbb{R}^{N}).  \label{2.7''}
\end{equation}%
If in addition $v=J_{p}(\phi )$ with $\phi \in \mathcal{D}_{AP}(\mathbb{R}%
^{N})/\mathbb{C}$ then $\left\langle \Delta _{p^{\prime }}w,J_{p}(\phi
)\right\rangle =-\left\langle D_{p^{\prime }}w,D_{p}\phi \right\rangle $, so
that, for $p=2$, we get 
\begin{equation*}
\left\langle \Delta _{2}w,J_{2}(\phi )\right\rangle =\left\langle w,\Delta
_{2}\phi \right\rangle \text{\ for all }w\in \mathcal{B}_{AP}^{2}(\mathbb{R}%
^{N})\text{ and }\phi \in \mathcal{D}_{AP}(\mathbb{R}^{N})/\mathbb{C}\text{.}
\end{equation*}

The following result is also immediate.

\begin{proposition}
\label{p3.5}For $u\in AP^{\infty }(\mathbb{R}^{N})$ we have 
\begin{equation*}
\Delta _{p}\varrho (u)=\varrho (\Delta _{y}u)
\end{equation*}%
where $\Delta _{y}$ stands for the usual Laplacian operator on $\mathbb{R}%
_{y}^{N}$.
\end{proposition}

We end this subsection with some notations. Let $f\in \mathcal{B}_{AP}^{p}(%
\mathbb{R}^{N})$. We know that $D^{\alpha _{i}}f$ exists (in the sense of
distributions) and that $D^{\alpha _{i}}f=D_{i,p}f$ if $f\in \mathcal{D}%
_{i,p}$. So we can drop the subscript $p$ and henceforth denote $D_{i,p}$
(resp. $\partial _{i,p}$) by $\overline{\partial }/\partial y_{i}$ (resp. $%
\partial _{i}$). Thus, $\overline{D}_{y}$ will stand for the gradient
operator $(\overline{\partial }/\partial y_{i})_{1\leq i\leq N}$ and $%
\overline{\text{div}}_{y}$ for the divergence operator div$_{p}$. We will
also denote the operator $\overline{D}_{i,p}$ by $\overline{\partial }%
/\partial y_{i}$. Since $J_{p}$ is an embedding, this allows us to view $%
\mathcal{B}_{AP,\rho }^{1,p}(\mathbb{R}^{N})$ (and hence $\mathcal{D}%
_{AP,\rho }(\mathbb{R}^{N})$) as a dense subspace of $\mathcal{B}%
_{\#AP}^{1,p}(\mathbb{R}^{N})$. $D_{i,p}$ will therefore be seen as the
restriction of $\overline{D}_{i,p}$ to $\mathcal{B}_{AP,\rho }^{1,p}(\mathbb{%
R}^{N})$. Thus we will henceforth omit $J_{p}$ in the notation if it is
understood from the context and there is no risk of confusion. This will
lead to the notation $\overline{D}_{p}=\overline{D}_{y}=(\overline{\partial }%
/\partial y_{i})_{1\leq i\leq N}$ and $\partial _{p}=\partial =(\partial
_{i})_{1\leq i\leq N}$. Finally, we will denote the Laplacian operator on $%
\mathcal{B}_{AP}^{p}(\mathbb{R}^{N})$ by $\overline{\Delta }_{y}$.

\subsection{The $\Sigma $-convergence}

Let $A_{y}=AP(\mathbb{R}_{y}^{N})$ and $A_{\tau }=AP(\mathbb{R}_{\tau })$.
We know that $A=AP(\mathbb{R}_{y,\tau }^{N+1})$ is the closure in $\mathcal{B%
}(\mathbb{R}_{y,\tau }^{N+1})$ of the tensor product $A_{y}\otimes A_{\tau }$%
. We denote by $\mathcal{K}_{y}$ (resp. $\mathcal{K}_{\tau }$, $\mathcal{K}$%
) the spectrum of $A_{y}$ (resp. $A_{\tau }$, $A$). The same letter $%
\mathcal{G}$ will denote the Gelfand transformation on $A_{y}$, $A_{\tau }$
and $A$, as well. Points in $\mathcal{K}_{y}$ (resp. $\mathcal{K}_{\tau }$)
are denoted by $s$ (resp. $s_{0}$). The Haar measure on the compact group $%
\mathcal{K}_{y}$ (resp. $\mathcal{K}_{\tau }$) is denoted by $\beta _{y}$
(resp. $\beta _{\tau }$). We have $\mathcal{K}=\mathcal{K}_{y}\times 
\mathcal{K}_{\tau }$ (Cartesian product) and the Haar measure on $\mathcal{K}
$ is precisely the product measure $\beta =\beta _{y}\otimes \beta _{\tau }$%
; the last equality follows in an obvious way by the density of $%
A_{y}\otimes A_{\tau }$ in $A$ and by the Fubini's theorem. Finally, the
letter $E$ will throughout denote $(\varepsilon _{n})_{n\in \mathbb{N}}$
with $0<\varepsilon _{n}\leq 1$ and such that $\varepsilon _{n}\rightarrow 0$
as $n\rightarrow \infty $. In what follows, we use the same notation as in
the preceding section.

\begin{definition}
\label{d3.1}\emph{A sequence }$(u_{\varepsilon })_{\varepsilon >0}\subset
L^{p}(Q_{T})$\emph{\ (}$1\leq p<\infty $\emph{) is said to }weakly $\Sigma $%
-converge\emph{\ in }$L^{p}(Q_{T})$\emph{\ to some }$u_{0}\in L^{p}(Q_{T};%
\mathcal{B}_{AP}^{p}(\mathbb{R}_{y,\tau }^{N+1}))$\emph{\ if as }$%
\varepsilon \rightarrow 0$\emph{, we have } 
\begin{equation}
\begin{array}{l}
\int_{Q_{T}}u_{\varepsilon }(x,t)f\left( x,t,\frac{x}{\varepsilon },\frac{t}{%
\varepsilon ^{2}}\right) dxdt \\ 
\ \ \ \ \ \ \ \ \ \ \ \ \ \ \rightarrow \iint_{Q_{T}\times \mathcal{K}}%
\widehat{u}_{0}(x,t,s,s_{0})\widehat{f}(x,t,s,s_{0})dxdtd\beta%
\end{array}
\label{3.1}
\end{equation}%
\emph{for every }$f\in L^{p^{\prime }}(Q_{T};A)$\emph{\ (}$1/p^{\prime
}=1-1/p$\emph{), where }$\widehat{u}_{0}=\mathcal{G}_{1}\circ u_{0}$\emph{\
and }$\widehat{f}=\mathcal{G}_{1}\circ (\varrho \circ f)=\mathcal{G}\circ f$%
\emph{. We express this by writing} $u_{\varepsilon }\rightarrow u_{0}$ in $%
L^{p}(Q_{T})$-weak $\Sigma $.
\end{definition}

We recall some important results whose proofs can be founded in \cite[%
Theorems 3.1 and 3.5]{CMP} (see also \cite[Theorems 3.1 and 3.6]{NA}).

\begin{theorem}
\label{t3.1}Let $1<p<\infty $. Let $(u_{\varepsilon })_{\varepsilon \in E}$
be a bounded sequence in $L^{p}(Q_{T})$. Then there exists a subsequence $%
E^{\prime }$ from $E$ such that the sequence $(u_{\varepsilon
})_{\varepsilon \in E^{\prime }}$ is weakly $\Sigma $-convergent in $%
L^{p}(Q_{T})$.
\end{theorem}

The next result is of capital interest in the homogenization process.

\begin{theorem}
\label{t3.2}Let $1<p<\infty $. Let $(u_{\varepsilon })_{\varepsilon \in E}$
be a bounded sequence in $L^{p}(0,T;W_{0}^{1,p}(Q))$. Then there exist a
subsequence $E^{\prime }$ of $E$ and a couple of functions 
\begin{equation*}
(u_{0},u_{1})\in L^{p}(0,T;W_{0}^{1,p}(Q))\times L^{p}(Q_{T};\mathcal{B}%
_{AP}^{p}(\mathbb{R}_{\tau };\mathcal{B}_{\#AP}^{1,p}(\mathbb{R}_{y}^{N})))
\end{equation*}%
such that, as $E^{\prime }\ni \varepsilon \rightarrow 0$, 
\begin{equation}
u_{\varepsilon }\rightarrow u_{0}\ \text{in }L^{p}(0,T;W_{0}^{1,p}(Q))\text{%
-weak;}  \label{3.2}
\end{equation}%
\begin{equation}
\frac{\partial u_{\varepsilon }}{\partial x_{i}}\rightarrow \frac{\partial
u_{0}}{\partial x_{i}}+\frac{\overline{\partial }u_{1}}{\partial y_{i}}\text{%
\ in }L^{p}(Q_{T})\text{-weak }\Sigma \text{, }1\leq i\leq N.  \label{3.3}
\end{equation}
\end{theorem}

We shall also deal with the product of sequences. In this respect, we give a
further

\begin{definition}
\label{d3.2}\emph{A sequence }$(u_{\varepsilon })_{\varepsilon >0}\subset
L^{p}(Q_{T})$\emph{\ (}$1\leq p<\infty $\emph{) is said to }strongly $\Sigma 
$-converge\emph{\ in }$L^{p}(Q_{T})$\emph{\ to some }$u_{0}\in L^{p}(Q_{T};%
\mathcal{B}_{AP}^{p}(\mathbb{R}_{y,\tau }^{N+1}))$\emph{\ if it is weakly }$%
\Sigma $\emph{-convergent towards }$u_{0}$\emph{\ and further satisfies the
following condition: }%
\begin{equation}
\left\Vert u_{\varepsilon }\right\Vert _{L^{p}(Q_{T})}\rightarrow \left\Vert 
\widehat{u}_{0}\right\Vert _{L^{p}(Q_{T}\times \mathcal{K})}.  \label{3.12}
\end{equation}%
\emph{We denote this by writing }$u_{\varepsilon }\rightarrow u_{0}$\emph{\
in }$L^{p}(Q_{T})$\emph{-strong }$\Sigma $\emph{.}
\end{definition}

This being so we have the following.

\begin{theorem}[{\protect\cite[Theorem 6]{DPDE}}]
\label{t3.3}Let $1<p,q<\infty $ and $r\geq 1$ be such that $1/r=1/p+1/q\leq
1 $. Assume $(u_{\varepsilon })_{\varepsilon \in E}\subset L^{q}(Q_{T})$ is
weakly $\Sigma $-convergent in $L^{q}(Q_{T})$ to some $u_{0}\in L^{q}(Q_{T};%
\mathcal{B}_{AP}^{q}(\mathbb{R}_{y,\tau }^{N+1}))$, and $(v_{\varepsilon
})_{\varepsilon \in E}\subset L^{p}(Q_{T})$ is strongly $\Sigma $-convergent
in $L^{p}(Q_{T})$ to some $v_{0}\in L^{p}(Q_{T};\mathcal{B}_{AP}^{p}(\mathbb{%
R}_{y,\tau }^{N+1}))$. Then the sequence $(u_{\varepsilon }v_{\varepsilon
})_{\varepsilon \in E}$ is weakly $\Sigma $-convergent in $L^{r}(Q_{T})$ to $%
u_{0}v_{0}$.
\end{theorem}

As a consequence of the above theorem the following holds.

\begin{corollary}
\label{c3.1}Let $(u_{\varepsilon })_{\varepsilon \in E}\subset L^{p}(Q_{T})$
and $(v_{\varepsilon })_{\varepsilon \in E}\subset L^{p^{\prime
}}(Q_{T})\cap L^{\infty }(Q_{T})$ ($1<p<\infty $ and $p^{\prime }=p/(p-1)$)
be two sequences such that:

\begin{itemize}
\item[(i)] $u_{\varepsilon }\rightarrow u_{0}$ in $L^{p}(Q_{T})$-weak $%
\Sigma $;

\item[(ii)] $v_{\varepsilon }\rightarrow v_{0}$ in $L^{p^{\prime }}(Q_{T})$%
-strong $\Sigma $;

\item[(iii)] $(v_{\varepsilon })_{\varepsilon \in E}$ is bounded in $%
L^{\infty }(Q_{T})$.
\end{itemize}

\noindent Then $u_{\varepsilon }v_{\varepsilon }\rightarrow u_{0}v_{0}$ in $%
L^{p}(Q_{T})$-weak $\Sigma $.
\end{corollary}

Regarding the reiterated $\Sigma$-convergence we refer the reader to \cite%
{Douanla1, CPAA}.

\section{Homogenization results}

Throughout this section we make the following assumption on the functions
involved in (\ref{0.1}):

\begin{itemize}
\item[(\textbf{A5})] \textbf{Almost periodicity}. We assume that the
functions $a_{ij}$ and $b$ lie in $B_{AP}^{2}(\mathbb{R}_{y,\tau
}^{N+1})\cap L^{\infty }(\mathbb{R}_{y,\tau }^{N+1})$ for all $1\leq i,j\leq
N$. We also assume that the function $\rho $ belongs to $B_{AP}^{2}(\mathbb{R%
}^{N})\cap L^{\infty }(\mathbb{R}^{N})$ with $\mathfrak{M}_{y}(\rho )>0$.
Finally the function $\tau \mapsto f(\tau ,r)$ lies in $(AP(\mathbb{R}_{\tau
}))^{N}$ for any $r\in \mathbb{R}^{N}$.
\end{itemize}

\subsection{Preliminary results}

Let $1+\frac{2N}{N+2}\leq p<\infty $. It is a fact that the topological dual
of $\mathcal{B}_{AP}^{p}(\mathbb{R}_{\tau };\mathcal{B}_{\#AP}^{1,p}(\mathbb{%
R}_{y}^{N}))$ is $\mathcal{B}_{AP}^{p^{\prime }}(\mathbb{R}_{\tau };[%
\mathcal{B}_{\#AP}^{1,p}(\mathbb{R}_{y}^{N})]^{\prime })$; this can be
easily seen from the fact that $\mathcal{B}_{\#AP}^{1,p}(\mathbb{R}_{y}^{N})$
is reflexive (see Section 2) and $\mathcal{B}_{AP}^{p}(\mathbb{R}_{\tau };%
\mathcal{B}_{\#AP}^{1,p}(\mathbb{R}_{y}^{N}))$ is isometrically isomorphic
to $L^{p}(\mathcal{K}_{\tau };\mathcal{B}_{\#AP}^{1,p}(\mathbb{R}_{y}^{N}))$%
. We denote by $\left\langle ,\right\rangle $ (resp. $[,]$) the duality
pairing between $\mathcal{B}_{\#AP}^{1,p}(\mathbb{R}_{y}^{N})$ (resp. $%
\mathcal{B}_{AP}^{p}(\mathbb{R}_{\tau };\mathcal{B}_{\#AP}^{1,p}(\mathbb{R}%
_{y}^{N}))$) and $[\mathcal{B}_{\#AP}^{1,p}(\mathbb{R}_{y}^{N})]^{\prime }$
(resp. $\mathcal{B}_{AP}^{p^{\prime }}(\mathbb{R}_{\tau };[\mathcal{B}%
_{\#AP}^{1,p}(\mathbb{R}_{y}^{N})]^{\prime })$). Hence, for $u\in \mathcal{B}%
_{AP}^{p^{\prime }}(\mathbb{R}_{\tau };[\mathcal{B}_{\#AP}^{1,p}(\mathbb{R}%
_{y}^{N})]^{\prime })$ and $v\in \mathcal{B}_{AP}^{p}(\mathbb{R}_{\tau };%
\mathcal{B}_{\#AP}^{1,p}(\mathbb{R}_{y}^{N}))$, 
\begin{equation*}
\left[ u,v\right] =\int_{\mathcal{K}_{\tau }}\left\langle \widehat{u}(s_{0}),%
\widehat{v}(s_{0})\right\rangle d\beta _{\tau }(s_{0}).
\end{equation*}%
For a function $\psi \in \mathcal{D}_{AP,\rho }(\mathbb{R}^{N})$ we know
that $\psi $ expresses as follows: $\psi =\varrho _{y}(\psi _{1})$ with $%
\psi _{1}\in AP_{\rho }^{\infty }(\mathbb{R}_{y}^{N})=\{v\in AP^{\infty }(%
\mathbb{R}_{y}^{N}):\mathfrak{M}(\rho \psi _{1})=0\}$, where $\varrho _{y}$
denotes the canonical mapping of $B_{AP}^{p}(\mathbb{R}_{y}^{N})$ onto $%
\mathcal{B}_{AP}^{p}(\mathbb{R}_{y}^{N})$; see Section 2. We will refer to $%
\psi _{1}$ as the representative of $\psi $ in $AP^{\infty }(\mathbb{R}%
_{y}^{N})$. Likewise we define the representative of $\psi \in \mathcal{D}%
_{AP}(\mathbb{R}_{\tau })\otimes \mathcal{D}_{AP,\rho }(\mathbb{R}^{N})$ as
an element of $AP^{\infty }(\mathbb{R}_{\tau })\otimes AP_{\rho }^{\infty }(%
\mathbb{R}_{y}^{N})$ satisfying a similar property.

With all this in mind, we have the following

\begin{lemma}
\label{l5.1}Let $\psi \in \mathcal{C}_{0}^{\infty }(Q_{T})\otimes \lbrack 
\mathcal{D}_{AP}(\mathbb{R}_{\tau })\otimes \mathcal{D}_{AP,\rho }(\mathbb{R}%
_{y}^{N})]$ and $\psi _{1}$ be its representative in $\mathcal{C}%
_{0}^{\infty }(Q_{T})\otimes \lbrack AP^{\infty }(\mathbb{R}_{\tau })\otimes
AP_{\rho }^{\infty }(\mathbb{R}_{y}^{N})]$. Let $(u_{\varepsilon
})_{\varepsilon \in E}$, $E^{\prime }$ and $(u_{0},u_{1})$ be as in Theorem 
\emph{\ref{t3.2}}. Then, as $E^{\prime }\ni \varepsilon \rightarrow 0$, 
\begin{equation*}
\int_{Q_{T}}\frac{1}{\varepsilon }u_{\varepsilon }\rho ^{\varepsilon }\psi
_{1}^{\varepsilon }dxdt\rightarrow \int_{Q_{T}}\left[ \rho u_{1}(x,t),\psi
(x,t)\right] dxdt.
\end{equation*}
\end{lemma}

\begin{proof}
We recall that for $\psi _{1}$ as above, we have 
\begin{equation*}
\psi _{1}^{\varepsilon }(x,t)=\psi _{1}\left( x,t,\frac{x}{\varepsilon },%
\frac{t}{\varepsilon ^{2}}\right) \text{ for }(x,t)\in Q_{T}\text{.}
\end{equation*}%
This being so, since $\rho \psi _{1}(x,t,\cdot ,\tau )\in AP_{\rho }(\mathbb{%
R}_{y}^{N})=\{u\in AP(\mathbb{R}_{y}^{N}):\mathfrak{M}_{y}(\rho u)=0\}$,
there exists \cite{Jikov} a unique $\phi \in \mathcal{C}_{0}^{\infty
}(Q_{T})\otimes \lbrack AP^{\infty }(\mathbb{R}_{\tau })\otimes AP(\mathbb{R}%
_{y}^{N})]$ such that $\rho \psi _{1}=\Delta _{y}\phi $. The result follows
at once by the application of \cite[Lemma 3.4]{NgWou} (see also \cite%
{Douanla, Douanla2} for some periodic versions of this lemma and their
proofs).
\end{proof}

For $u\in \mathcal{B}_{AP}^{p}(\mathbb{R}_{\tau })$ we denote by $\overline{%
\partial }/\partial \tau $ the temporal derivative defined exactly as its
spatial counterpart $\overline{\partial }/\partial y_{i}$. We also put $%
\partial _{0}=\mathcal{G}_{1}(\overline{\partial }/\partial \tau )$. $%
\overline{\partial }/\partial \tau $ and $\partial _{0}$ enjoy the same
properties as $\overline{\partial }/\partial y_{i}$ (see Section 2). In
particular, they are skew adjoint. Now, let us view $\rho \overline{\partial 
}/\partial \tau $ as an unbounded operator defined from $\mathcal{U}=%
\mathcal{B}_{AP}^{p}(\mathbb{R}_{\tau };\mathcal{B}_{\#AP}^{1,p}(\mathbb{R}%
_{y}^{N}))$ into $\mathcal{U}^{\prime }=\mathcal{B}_{AP}^{p^{\prime }}(%
\mathbb{R}_{\tau };[\mathcal{B}_{\#AP}^{1,p}(\mathbb{R}_{y}^{N})]^{\prime })$%
. Similar to \cite[pp. 1243-1244]{EfendievPankov}, it gives rise to an
unbounded operator still denoted by $\rho \overline{\partial }/\partial \tau 
$ with the following properties:

\begin{itemize}
\item[(P)$_{1}$] The domain of $\rho \overline{\partial }/\partial \tau $ is 
$\mathcal{W}=\left\{ v\in \mathcal{U}:\rho \overline{\partial }v/\partial
\tau \in \mathcal{U}^{\prime }\right\} $;

\item[(P)$_{2}$] $\rho \overline{\partial }/\partial \tau $ is skew adjoint,
that is, for all $u,v\in \mathcal{W}$, 
\begin{equation*}
\left[ \rho \frac{\overline{\partial }v}{\partial \tau },u\right] =-\left[
\rho \frac{\overline{\partial }u}{\partial \tau },v\right] .
\end{equation*}

\item[(P)$_{3}$] The space $\mathcal{E}=\mathcal{D}_{AP}(\mathbb{R}_{\tau
})\otimes \mathcal{D}_{AP,\rho }(\mathbb{R}_{y}^{N})$ is dense in $\mathcal{W%
}$.
\end{itemize}

The above operator will be useful in the homogenization process. This being
so, the preceding lemma has an important corollary.

\begin{corollary}[{\protect\cite[Corollary 1]{ApplAnal}}]
\label{c5.1}Let the hypotheses be those of Lemma \emph{\ref{l5.1}}. Assume
moreover that $u_{1}\in \mathcal{W}$. Then, as $E^{\prime }\ni \varepsilon
\rightarrow 0$,%
\begin{equation*}
\int_{Q_{T}}\varepsilon u_{\varepsilon }\rho ^{\varepsilon }\frac{\partial
\psi _{1}^{\varepsilon }}{\partial t}dxdt\rightarrow -\int_{Q_{T}}\left[
\rho \frac{\overline{\partial }u_{1}}{\partial \tau }(x,t),\psi (x,t)\right]
dxdt.
\end{equation*}
\end{corollary}

Now assume that $\boldsymbol{u}_{1}=(u_{1}^{i})\in (\mathcal{B}%
_{\#AP}^{1,p})^{N}$ and $\boldsymbol{\psi }=(\psi ^{i})\in (\mathcal{D}_{AP}(%
\mathbb{R}_{\tau })\otimes \mathcal{D}_{AP,\rho }(\mathbb{R}_{y}^{N}))^{N}$.
Then we may still define $[\rho \frac{\overline{\partial }\boldsymbol{u}_{1}%
}{\partial \tau }(x,t),\boldsymbol{\psi }(x,t)]$, but this time as $%
\sum_{i=1}^{N}[\rho \frac{\overline{\partial }u_{1}^{i}}{\partial \tau }%
(x,t),\psi ^{i}(x,t)]$. We will use that notation in the sequel.

We end this subsection with one further result.

\begin{lemma}[{\protect\cite[Lemma 5]{ApplAnal}}]
\label{l4.4}Let $h:\mathbb{R}\times \mathbb{R}^{N}\rightarrow \mathbb{R}^{N}$
be a continuous function verifying the following conditions:

\begin{itemize}
\item[(i)] $\left\vert h(\tau ,r_{1})-h(\tau ,r_{2})\right\vert \leq
k\left\vert r_{1}-r_{2}\right\vert $ for any $\tau \in \mathbb{R}$ $\ $and
all $r_{1}$, $r_{2}\in \mathbb{R}^{N}$;

\item[(ii)] $h(\cdot ,r)\in AP(\mathbb{R})$ for all $r\in \mathbb{R}^{N}$.
\end{itemize}

\noindent Let $(\boldsymbol{u}_{\varepsilon })_{\varepsilon }$ be a sequence
in $L^{2}(Q_{T})^{N}$ such that $\boldsymbol{u}_{\varepsilon }\rightarrow 
\boldsymbol{u}_{0}$ in $L^{2}(Q_{T})^{N}$ as $\varepsilon \rightarrow 0$,
where $\boldsymbol{u}_{0}\in L^{2}(Q_{T})^{N}$. Then, setting $%
h^{\varepsilon }(\boldsymbol{u}_{\varepsilon })(x,t)=h(t/\varepsilon ^{2},%
\boldsymbol{u}_{\varepsilon }(x,t))$ we have, as $\varepsilon \rightarrow 0$%
, 
\begin{equation*}
h^{\varepsilon }(\boldsymbol{u}_{\varepsilon })\rightarrow h(\cdot ,%
\boldsymbol{u}_{0})\text{ in }L^{2}(Q_{T})^{N}\text{-weak }\Sigma \text{.}
\end{equation*}
\end{lemma}

\subsection{Homogenization results}

Before we can state the homogenization result for (\ref{0.1}) we need a few
notations. We begin by noting that the space $\mathcal{C}_{0}^{\infty
}(0,T)\otimes \mathcal{V}$ is dense in $L^{p}(0,T;V)$. Next, we introduce
the space 
\begin{equation*}
\mathcal{B}_{\Div}^{1,p}=\{\boldsymbol{u}\in (\mathcal{B}_{\#AP}^{1,p}(%
\mathbb{R}^{N}))^{N}:~\overline{\Div}_{y}\boldsymbol{u}=0\}
\end{equation*}%
where $\overline{\Div}_{y}\boldsymbol{u}=\sum_{i=1}^{N}\overline{\partial }%
u^{i}/\partial y_{i}$, and its smooth counterpart 
\begin{eqnarray*}
\mathcal{V}_{\Div} &=&\{\boldsymbol{u}\in (\mathcal{D}_{AP,\rho }(\mathbb{R}%
^{N}))^{N}:~\overline{\Div}_{y}\boldsymbol{u}=0\} \\
&\equiv &\{\boldsymbol{u}\in (\mathcal{D}_{AP,\rho }(\mathbb{R}^{N}))^{N}:%
\mathfrak{M}(\rho \boldsymbol{u})=0\text{ and }\overline{\Div}_{y}%
\boldsymbol{u}=0\}.
\end{eqnarray*}%
The following result holds.

\begin{lemma}
\label{l6.2}The space $\mathcal{V}_{\Div}$ is dense in $\mathcal{B}_{\Div%
}^{1,p}$.
\end{lemma}

\begin{proof}
This follows exactly in a same way as the proof of \cite[Lemma 2.3]{Wright1}.
\end{proof}

Now, let 
\begin{equation*}
\mathcal{W}_{\Div}=\left\{ u\in \mathcal{B}_{AP}^{p}(\mathbb{R}_{\tau };%
\mathcal{B}_{\Div}^{1,p}):\rho \frac{\partial u}{\partial \tau }\in \mathcal{%
B}_{AP}^{p^{\prime }}(\mathbb{R}_{\tau };(\mathcal{B}_{\Div}^{1,p})^{\prime
})\right\};
\end{equation*}%
We set 
\begin{equation*}
\mathbb{F}_{0}^{1,p}=L^{p}(0,T;V)\times L^{p}(Q_{T};\mathcal{W}_{\Div})
\end{equation*}%
and 
\begin{equation*}
\mathcal{F}_{0}^{\infty }=[\mathcal{C}_{0}^{\infty }(0,T)\otimes \mathcal{V}%
]\times \left[ \mathcal{C}_{0}^{\infty }(Q_{T})\otimes \left( \mathcal{D}%
_{AP}(\mathbb{R}_{\tau })\otimes \mathcal{V}_{\Div}\right) \right] .
\end{equation*}%
Thanks to Lemma \ref{l6.2} and to the density of $\mathcal{E}=\mathcal{D}%
_{AP}(\mathbb{R}_{\tau })\otimes \mathcal{D}_{AP,\rho }(\mathbb{R}_{y}^{N})$
in $\mathcal{W}$, we have the density of $\mathcal{F}_{0}^{\infty }$ in $%
\mathbb{F}_{0}^{1,p}$.

Let $(\boldsymbol{u}_{\varepsilon })_{\varepsilon \in E}$ be a sequence of
solution to (\ref{0.1}) (or to (\ref{0.8}))where we assume that $E$ is an
ordinary sequence tending to zero with $\varepsilon $. In view of
Proposition \ref{p2.2}, there are a subsequence $E^{\prime }$ of $E$ and a
function $\boldsymbol{u}_{0}\in L^{2}(Q_{T})^{N}$ such that, as $E^{\prime
}\ni \varepsilon \rightarrow 0$ 
\begin{equation}
\boldsymbol{u}_{\varepsilon }\rightarrow \boldsymbol{u}_{0}\text{ in }%
L^{2}(Q_{T})^{N}\text{.}  \label{6.23}
\end{equation}%
In view of (\ref{2.2}) and by the diagonal process, one can find a
subsequence of $(\boldsymbol{u}_{\varepsilon })_{\varepsilon \in E^{\prime
}} $ (not relabeled) which weakly converges in $L^{p}(0,T;V)$ to the
function $\boldsymbol{u}_{0}$ (this means that $\boldsymbol{u}_{0}\in
L^{p}(0,T;V)$). From Theorem \ref{t3.2}, we infer the existence of a
function $\boldsymbol{u}_{1}=(u_{1}^{k})_{1\leq k\leq N}\in L^{p}(Q_{T};%
\mathcal{B}_{AP}^{p}(\mathbb{R}_{\tau };\mathcal{B}_{\#AP}^{1,p})^{N})$ such
that the convergence result 
\begin{equation}
\frac{\partial \boldsymbol{u}_{\varepsilon }}{\partial x_{i}}\rightarrow 
\frac{\partial \boldsymbol{u}_{0}}{\partial x_{i}}+\frac{\overline{\partial }%
\boldsymbol{u}_{1}}{\partial y_{i}}\text{ in }L^{p}(Q_{T})^{N}\text{-weak }%
\Sigma \text{ }(1\leq i\leq N)  \label{6.24}
\end{equation}%
holds when $E^{\prime }\ni \varepsilon \rightarrow 0$. Owing to Lemma \ref%
{l2.0'} (see (\ref{2.15}) therein) there exist a subsequence of $E^{\prime }$
(still denoted by $E^{\prime }$) and a function $q\in L^{p^{\prime }}(Q_{T};%
\mathcal{B}_{AP}^{p^{\prime }}(\mathbb{R}^{N+1}))$ with $\int_{Q}qdx=0$ such
that 
\begin{equation}
q_{\varepsilon }\rightarrow q\text{ in }L^{p^{\prime }}(Q_{T})\text{-weak }%
\Sigma \text{ as }E^{\prime }\ni \varepsilon \rightarrow 0.  \label{6.25}
\end{equation}%
We recall that $\frac{\partial \boldsymbol{u}_{0}}{\partial x_{i}}=\left( 
\frac{\partial u_{0}^{k}}{\partial x_{i}}\right) _{1\leq k\leq N}$ ($%
\boldsymbol{u}_{0}=(u_{0}^{k})_{1\leq k\leq N}$) and $\frac{\overline{%
\partial }\boldsymbol{u}_{1}}{\partial y_{i}}=\left( \frac{\overline{%
\partial }u_{1}^{k}}{\partial y_{i}}\right) _{1\leq k\leq N}$. Now, let us
consider the following functionals: 
\begin{eqnarray*}
\widehat{a}_{I}(\boldsymbol{u},\boldsymbol{v})
&=&\sum_{i,j,k=1}^{N}\iint_{Q_{T}\times \mathcal{K}}\widehat{a}_{ij}(s,s_{0})%
\mathbb{D}_{j}u^{k}\mathbb{D}_{i}v^{k}dxdtd\beta \\
&&\ \ \ +\sum_{i,j=1}^{N}\iint_{Q_{T}\times \mathcal{K}}\widehat{b}%
(s,s_{0})\left\vert \mathbb{D}\boldsymbol{u}\right\vert ^{p-2}\mathbb{D}%
_{i}u^{j}\mathbb{D}_{i}v^{j}dxdtd\beta
\end{eqnarray*}%
where $\mathbb{D}_{j}u^{k}=\frac{\partial u_{0}^{k}}{\partial x_{j}}%
+\partial _{j}\widehat{u}_{1}^{k}$ ($\partial _{j}\widehat{u}_{1}^{k}=%
\mathcal{G}_{1}\left( \frac{\overline{\partial }u_{1}^{k}}{\partial y_{j}}%
\right) $, and the same definition for $\mathbb{D}_{i}v^{k}$) and $\mathbb{D}%
\boldsymbol{u}=(\mathbb{D}_{j}\boldsymbol{u})_{1\leq j\leq N}$ with $\mathbb{%
D}_{j}\boldsymbol{u}=(\mathbb{D}_{j}u^{k})_{1\leq k\leq N}$; 
\begin{equation*}
\widehat{b}_{I}(\boldsymbol{u},\boldsymbol{v},\boldsymbol{w}%
)=\sum_{i,k=1}^{N}\iint_{Q_{T}}u_{0}^{i}\frac{\partial v_{0}^{k}}{\partial
x_{i}}w_{0}^{k}dxdt
\end{equation*}%
for $\boldsymbol{u}=(\boldsymbol{u}_{0},\boldsymbol{u}_{1}),\boldsymbol{v}=(%
\boldsymbol{v}_{0},\boldsymbol{v}_{1}),\boldsymbol{w}=(\boldsymbol{w}_{0},%
\boldsymbol{w}_{1})\in \mathbb{F}_{0}^{1,p}$. The functionals $\widehat{a}%
_{I}$ and $\widehat{b}_{I}$ are well-defined. Next, associated to these
functionals is the variational problem 
\begin{equation}
\left\{ 
\begin{array}{l}
\boldsymbol{u}=(\boldsymbol{u}_{0},\boldsymbol{u}_{1})\in \mathbb{F}%
_{0}^{1,p}: \\ 
-\mathfrak{M}_{y}(\rho )\int_{Q_{T}}\boldsymbol{u}_{0}\cdot \boldsymbol{\psi 
}_{0}^{\prime }dxdt-\int_{Q_{T}}\left[ \rho \boldsymbol{u}_{1},\frac{%
\partial \boldsymbol{\psi }_{1}}{\partial \tau }\right] dxdt+\widehat{a}_{I}(%
\boldsymbol{u},\boldsymbol{\Phi })+\widehat{b}_{I}(\boldsymbol{u},%
\boldsymbol{u},\boldsymbol{\Phi }) \\ 
\ \ \ \ =\int_{Q_{T}}\mathfrak{M}(\rho f(\cdot ,\boldsymbol{u}_{0}))\cdot 
\boldsymbol{\psi }_{0}dxdt\text{ for all }\Phi =(\boldsymbol{\psi }_{0},%
\boldsymbol{\psi }_{1})\in \mathcal{F}_{0}^{\infty }\text{.}%
\end{array}%
\right.  \label{6.27}
\end{equation}

The following \textit{global} homogenization result holds.

\begin{theorem}
\label{t6.3}The couple $(\boldsymbol{u}_{0},\boldsymbol{u}_{1})$ determined
by \emph{(\ref{6.23})-(\ref{6.24})} solves problem \emph{(\ref{6.27})}.
\end{theorem}

\begin{proof}
From the equality $\Div\boldsymbol{u}_{\varepsilon }=0$ we easily derive $%
\Div\boldsymbol{u}_{0}=0$ and $\overline{\Div}_{y}\boldsymbol{u}_{1}=0$. In
order to show that $\boldsymbol{u}=(\boldsymbol{u}_{0},\boldsymbol{u}%
_{1})\in \mathbb{F}_{0}^{1,p}$ we need to verify that $\partial \boldsymbol{u%
}_{1}(x,t)/\partial \tau \in \mathcal{B}_{AP}^{p^{\prime }}(\mathbb{R}_{\tau
};(\mathcal{B}_{\Div}^{1,p})^{\prime })$ for a.e. $(x,t)\in Q_{T}$. But this
will be done later. Now, we first need to show that $\boldsymbol{u}$ solve (%
\ref{6.27}). For that, let $\boldsymbol{\Phi }=(\boldsymbol{\psi }_{0},%
\boldsymbol{\psi }_{1})$ with $\boldsymbol{\psi }_{0}\in \mathcal{C}%
_{0}^{\infty }(Q_{T})^{N}$ and $\boldsymbol{\psi }_{1}=\varrho _{y}^{N}(%
\boldsymbol{\psi })\equiv (\varrho _{y}(\psi _{j}))_{1\leq j\leq N}\in 
\mathcal{C}_{0}^{\infty }(Q_{T})\otimes (\mathcal{D}_{AP}(\mathbb{R}_{\tau
})\otimes \mathcal{V}_{\text{div}}$), $\varrho _{y}$ being the canonical
mapping from $B_{AP}^{p}(\mathbb{R}_{y}^{N})$ into $\mathcal{B}_{AP}^{p}(%
\mathbb{R}_{y}^{N})$; define $\boldsymbol{\Phi }_{\varepsilon }$ as follows: 
\begin{equation*}
\boldsymbol{\Phi }_{\varepsilon }(x,t)=\boldsymbol{\psi }_{0}(x,t)+%
\varepsilon \boldsymbol{\psi }\left( x,t,\frac{x}{\varepsilon },\frac{t}{%
\varepsilon ^{2}}\right) \text{ for }(x,t)\in Q_{T}\text{.}
\end{equation*}%
Then we have $\boldsymbol{\Phi }_{\varepsilon }\in \mathcal{C}_{0}^{\infty
}(Q_{T})^{N}$ and, using $\boldsymbol{\Phi }_{\varepsilon }$ as a test
function in the variational formulation of (\ref{0.1}) we get 
\begin{eqnarray}
&&-\int_{Q_{T}}\rho ^{\varepsilon }\boldsymbol{u}_{\varepsilon }\cdot \frac{%
\partial \boldsymbol{\Phi }_{\varepsilon }}{\partial t}dxdt+\int_{Q_{T}}a^{%
\varepsilon }D\boldsymbol{u}_{\varepsilon }\cdot D\boldsymbol{\Phi }%
_{\varepsilon }dxdt+\int_{0}^{T}b_{I}(\boldsymbol{u}_{\varepsilon },%
\boldsymbol{u}_{\varepsilon },\boldsymbol{\Phi }_{\varepsilon })dt
\label{6.28} \\
&&+\int_{Q_{T}}b^{\varepsilon }\left\vert D\boldsymbol{u}_{\varepsilon
}\right\vert ^{p-2}D\boldsymbol{u}_{\varepsilon }\cdot D\boldsymbol{\Phi }%
_{\varepsilon }dxdt-\int_{Q_{T}}q_{\varepsilon }\Div\boldsymbol{\Phi }%
_{\varepsilon }dxdt  \notag \\
&=&\int_{0}^{T}\left( \rho ^{\varepsilon }f^{\varepsilon }(t,\boldsymbol{u}%
_{\varepsilon }(t)),\boldsymbol{\Phi }_{\varepsilon }\right) dt.  \notag
\end{eqnarray}%
We pass to the limit in (\ref{6.28}) by considering each term separately.
First we have 
\begin{equation*}
\int_{Q_{T}}\rho ^{\varepsilon }\boldsymbol{u}_{\varepsilon }\cdot \frac{%
\partial \boldsymbol{\Phi }_{\varepsilon }}{\partial t}dxdt=\int_{Q_{T}}\rho
^{\varepsilon }\boldsymbol{u}_{\varepsilon }\cdot \frac{\partial \boldsymbol{%
\psi }_{0}}{\partial t}dxdt+\varepsilon \int_{Q_{T}}\rho ^{\varepsilon }%
\boldsymbol{u}_{\varepsilon }\cdot \frac{\partial \boldsymbol{\psi }%
^{\varepsilon }}{\partial t}dxdt
\end{equation*}%
In view of \ Corollary \ref{c5.1} we have that 
\begin{equation*}
\varepsilon \int_{Q_{T}}\rho ^{\varepsilon }\boldsymbol{u}_{\varepsilon
}\cdot \frac{\partial \boldsymbol{\psi }^{\varepsilon }}{\partial t}%
dxdt\rightarrow \int_{Q_{T}}\left[ \rho \boldsymbol{u}_{1},\frac{\partial 
\boldsymbol{\psi }_{1}}{\partial \tau }\right] dxdt\text{ when }E^{\prime
}\ni \varepsilon \rightarrow 0.
\end{equation*}%
On the other hand, due to (\ref{6.23}) we get 
\begin{equation*}
\int_{Q_{T}}\rho ^{\varepsilon }\boldsymbol{u}_{\varepsilon }\cdot \frac{%
\partial \boldsymbol{\psi }_{0}}{\partial t}dxdt\rightarrow \mathfrak{M}%
_{y}(\rho )\int_{Q_{T}}\boldsymbol{u}_{0}\cdot \frac{\partial \boldsymbol{%
\psi }_{0}}{\partial t}dxdt\text{ as }E^{\prime }\ni \varepsilon \rightarrow
0.
\end{equation*}%
Hence, as $E^{\prime }\ni \varepsilon \rightarrow 0$, 
\begin{equation*}
\int_{Q_{T}}\boldsymbol{u}_{\varepsilon }\cdot \frac{\partial \boldsymbol{%
\Phi }_{\varepsilon }}{\partial t}dxdt\rightarrow \mathfrak{M}_{y}(\rho
)\int_{Q_{T}}\boldsymbol{u}_{0}\cdot \frac{\partial \boldsymbol{\psi }_{0}}{%
\partial t}dxdt+\int_{Q_{T}}\left[ \rho \boldsymbol{u}_{1},\frac{\partial 
\boldsymbol{\psi }_{1}}{\partial \tau }\right] dxdt.
\end{equation*}

Next, it is an usual well known fact that (see e.g., \cite[Lemma 3.5]%
{Douanla1}), the convergence result (\ref{6.24}) together with the weak $%
\Sigma $-convergence of the sequence $(D\boldsymbol{\Phi }_{\varepsilon
})_{\varepsilon }$ to $\mathbb{D}\boldsymbol{\Phi }$, imply 
\begin{equation*}
\int_{Q_{T}}a^{\varepsilon }D\boldsymbol{u}_{\varepsilon }\cdot D\boldsymbol{%
\Phi }_{\varepsilon }dxdt\rightarrow \sum_{i,j,k=1}^{N}\iint_{Q_{T}\times 
\mathcal{K}}\widehat{a}_{ij}(s,s_{0})\mathbb{D}_{j}u^{k}\mathbb{D}%
_{i}v^{k}dxdtd\beta .
\end{equation*}%
Considering the next term, we use the monotonicity property to have 
\begin{equation}
\int_{Q_{T}}(b^{\varepsilon }\left\vert D\boldsymbol{u}_{\varepsilon
}\right\vert ^{p-2}D\boldsymbol{u}_{\varepsilon }-b^{\varepsilon }\left\vert
D\boldsymbol{\Phi }_{\varepsilon }\right\vert ^{p-2}D\boldsymbol{\Phi }%
_{\varepsilon })\cdot (D\boldsymbol{u}_{\varepsilon }-D\boldsymbol{\Phi }%
_{\varepsilon })dxdt\geq 0\text{.}  \label{6.29}
\end{equation}%
Owing to the estimate (\ref{2.2}) we infer that 
\begin{equation*}
\sup_{\varepsilon >0}\left\Vert b^{\varepsilon }\left\vert D\boldsymbol{u}%
_{\varepsilon }\right\vert ^{p-2}D\boldsymbol{u}_{\varepsilon }\right\Vert
_{L^{p^{\prime }}(Q_{T})^{N\times N}}^{p^{\prime }}<\infty ,
\end{equation*}%
so that, from Theorem \ref{t3.1}, there exist a function $\chi \in
L^{p^{\prime }}(Q_{T};\mathcal{B}_{AP}^{p^{\prime }}(\mathbb{R}%
^{N}))^{N\times N}$ and a subsequence of $E^{\prime }$ not relabeled, such
that $b^{\varepsilon }\left\vert D\boldsymbol{u}_{\varepsilon }\right\vert
^{p-2}D\boldsymbol{u}_{\varepsilon }\rightarrow \chi $ in $L^{p^{\prime
}}(Q_{T})^{N\times N}$-weak $\Sigma $ as $E^{\prime }\ni \varepsilon
\rightarrow 0$. We therefore argue as in the proof of \cite[Theorem 3.1]%
{NgWou1} (see also \cite[Theorem 3.2]{CPAA}) to pass to the limit in (\ref%
{6.29}) (as $E^{\prime }\ni \varepsilon \rightarrow 0$) and get 
\begin{equation}
\iint_{Q_{T}\times \mathcal{K}}(\widehat{\chi }-\widehat{b}\left\vert 
\mathbb{D}\boldsymbol{\Phi }\right\vert ^{p-2}\mathbb{D}\Phi ))\cdot (%
\mathbb{D}\boldsymbol{u}-\mathbb{D}\boldsymbol{\Phi })dxdtd\beta \geq 0
\label{6.30}
\end{equation}%
for any $\Phi \in \mathcal{F}_{0}^{\infty }$ where, as above, $\mathbb{D}%
\boldsymbol{u}=D\boldsymbol{u}_{0}+\partial \widehat{\boldsymbol{u}}_{1}$ ($%
\boldsymbol{u}=(\boldsymbol{u}_{0},\boldsymbol{u}_{1})$) and $\mathbb{D}%
\boldsymbol{\Phi }=D\boldsymbol{\psi }_{0}+\partial \widehat{\boldsymbol{%
\psi }}_{1}$. By the density of $\mathcal{F}_{0}^{\infty }$ in $\mathbb{F}%
_{0}^{1,p}$ and by a continuity argument, (\ref{6.30}) still holds for $%
\boldsymbol{\Phi }\in \mathbb{F}_{0}^{1,p}$. Hence by taking $\boldsymbol{%
\Phi }=\boldsymbol{u}+\lambda \boldsymbol{v}$ for $\boldsymbol{v}=(%
\boldsymbol{v}_{0},\boldsymbol{v}_{1})\in \mathbb{F}_{0}^{1,p}$ and $\lambda
>0$ arbitrarily fixed, we get 
\begin{equation*}
\lambda \iint_{Q_{T}\times \mathcal{K}}(\widehat{\chi }-\widehat{b}(\cdot ,%
\mathbb{D}\boldsymbol{u}+\lambda \mathbb{D}\boldsymbol{v}))\cdot \mathbb{D}%
\boldsymbol{v}dxdtd\bar{\mathbb{P}}d\beta \geq 0\;\text{for all }\boldsymbol{%
v}\in \mathbb{F}_{0}^{1,p}.
\end{equation*}%
Therefore by a mere routine, we deduce that $\chi =b\left\vert D\boldsymbol{u%
}_{0}+\overline{D}_{y}\boldsymbol{u}_{1}\right\vert ^{p-2}(D\boldsymbol{u}%
_{0}+\overline{D}_{y}\boldsymbol{u}_{1})$.

The next point to check is to compute the $\lim_{\varepsilon \rightarrow
0}\int_{0}^{T}b_{I}(\boldsymbol{u}_{\varepsilon },\boldsymbol{u}%
_{\varepsilon },\boldsymbol{\Phi }_{\varepsilon })dt$. We claim that, as $%
\varepsilon \rightarrow 0$, 
\begin{equation*}
\int_{0}^{T}b_{I}(\boldsymbol{u}_{\varepsilon },\boldsymbol{u}_{\varepsilon
},\boldsymbol{\Phi }_{\varepsilon })dt\rightarrow \widehat{b}_{I}(%
\boldsymbol{u},\boldsymbol{u},\boldsymbol{\Phi }).
\end{equation*}%
Indeed, by the strong convergence result (\ref{6.23}) in conjunction with
the Theorem \ref{t3.3}, our claim is justified.

We obviously have that 
\begin{eqnarray*}
\int_{0}^{T}\left( \rho ^{\varepsilon }f^{\varepsilon }(t,\boldsymbol{u}%
_{\varepsilon }(t)),\boldsymbol{\Phi }_{\varepsilon }\right) dt
&=&\int_{Q_{T}}\rho ^{\varepsilon }f^{\varepsilon }(t,\boldsymbol{u}%
_{\varepsilon }(t))\cdot \boldsymbol{\Phi }_{\varepsilon }dxdt \\
&=&\int_{Q_{T}}\rho ^{\varepsilon }f^{\varepsilon }(t,\boldsymbol{u}%
_{\varepsilon }(t))\cdot \boldsymbol{\psi }_{0}dxdt \\
&&+\varepsilon \int_{Q_{T}}\rho ^{\varepsilon }f^{\varepsilon }(t,%
\boldsymbol{u}_{\varepsilon }(t))\cdot \boldsymbol{\psi }^{\varepsilon }dxdt.
\end{eqnarray*}%
It can easily be seen that the second integral of the right-hand side goes
to zero with $\varepsilon \in E^{\prime }$. For the first one, taking $\rho
^{\varepsilon }\boldsymbol{\psi }_{0}$ as a test function and using Lemma %
\ref{l4.4}, we are led to 
\begin{eqnarray*}
\int_{Q_{T}}\rho ^{\varepsilon }f^{\varepsilon }(\cdot ,\boldsymbol{u}%
_{\varepsilon })\cdot \boldsymbol{\psi }_{0}dxdt &\rightarrow
&\iint_{Q_{T}\times \mathcal{K}}\widehat{\rho }\widehat{f}(\cdot ,%
\boldsymbol{u}_{0})\cdot \boldsymbol{\psi }_{0}dxdtd\beta \\
&=&\int_{Q_{T}}\mathfrak{M}(\rho f(\cdot ,\boldsymbol{u}_{0}))\cdot 
\boldsymbol{\psi }_{0}dxdt.
\end{eqnarray*}%
Now, on the basis of (\ref{6.25}) there is no difficulty in showing that 
\begin{equation*}
\int_{Q_{T}}q_{\varepsilon }\Div\boldsymbol{\Phi }_{\varepsilon
}dxdt\rightarrow \iint_{Q_{T}\times \mathcal{K}}\widehat{q}\Div\boldsymbol{%
\psi }_{0}dxdtd\beta .
\end{equation*}%
It emerges from the above study that $\boldsymbol{u}=(\boldsymbol{u}_{0},%
\boldsymbol{u}_{1})$ and $q_{0}=\mathfrak{M}(q):=\int_{\mathcal{K}}\widehat{q%
}d\beta $ satisfies 
\begin{equation}
\begin{array}{l}
-\mathfrak{M}_{y}(\rho )\int_{Q_{T}}\boldsymbol{u}_{0}\cdot \boldsymbol{\psi 
}_{0}^{\prime }dxdt-\int_{Q_{T}}\left[ \rho \boldsymbol{u}_{1},\frac{%
\partial \boldsymbol{\psi }_{1}}{\partial \tau }\right] dxdt+\widehat{a}_{I}(%
\boldsymbol{u},\boldsymbol{\Phi })+\widehat{b}_{I}(\boldsymbol{u},%
\boldsymbol{u},\boldsymbol{\Phi }) \\ 
-\int_{Q_{T}}q_{0}\Div\boldsymbol{\psi }_{0}dxdt=\int_{Q_{T}}\mathfrak{M}%
(\rho f(\cdot ,\boldsymbol{u}_{0}))\cdot \boldsymbol{\psi }_{0}dxdt%
\end{array}
\label{6.26}
\end{equation}%
for all $\boldsymbol{\Phi }=(\boldsymbol{\psi }_{0},\boldsymbol{\psi }_{1})$
as above. Choosing in particular $\boldsymbol{\Phi }\in \mathcal{F}%
_{0}^{\infty }$ we end up with $(\boldsymbol{u}_{0},\boldsymbol{u}_{1})$
solving (\ref{6.27}) since $\Div\boldsymbol{\psi }_{0}=0$ in that case.
\end{proof}

Problem (\ref{6.27}) is the \emph{global homogenized problem}. In order to
derive the homogenized problem, we need to uncouple problem (\ref{6.27}). As
we can see, (\ref{6.27}) is equivalent to the following system (\ref{6.36})-(%
\ref{6.37}): 
\begin{equation}
\left\{ 
\begin{array}{l}
-\int_{Q_{T}}\left[ \rho \boldsymbol{u}_{1},\frac{\partial \boldsymbol{\psi }%
_{1}}{\partial \tau }\right] dxdt+\iint_{Q_{T}\times \mathcal{K}}\widehat{a}%
\mathbb{D}\boldsymbol{u}\cdot \partial \widehat{\boldsymbol{\psi }}%
_{1}dxdtd\beta \\ 
+\iint_{Q_{T}\times \mathcal{K}}\widehat{b}\left\vert \mathbb{D}\boldsymbol{u%
}\right\vert ^{p-2}\mathbb{D}\boldsymbol{u}\cdot \partial \widehat{%
\boldsymbol{\psi }}_{1}dxdtd\beta =0\text{\ for all }\boldsymbol{\psi }%
_{1}\in \mathcal{C}_{0}^{\infty }(Q_{T})\otimes \lbrack \mathcal{D}_{A_{\tau
}}(\mathbb{R}_{\tau })\otimes \mathcal{V}_{\Div}]\text{;}%
\end{array}%
\right.  \label{6.36}
\end{equation}%
\begin{equation}
\left\{ 
\begin{array}{l}
-\mathfrak{M}_{y}(\rho )\int_{Q_{T}}\boldsymbol{u}_{0}\cdot \boldsymbol{\psi 
}_{0}^{\prime }dxdt+\widehat{a}_{I}(\boldsymbol{u},(\boldsymbol{\psi }%
_{0},0))+\widehat{b}_{I}(\boldsymbol{u},\boldsymbol{u},(\boldsymbol{\psi }%
_{0},0)) \\ 
\ \ \ \ =\int_{Q_{T}}\mathfrak{M}(\rho f(\cdot ,\boldsymbol{u}_{0}))\cdot 
\boldsymbol{\psi }_{0}dxdt\text{ for all }\boldsymbol{\psi }_{0}\in \mathcal{%
C}_{0}^{\infty }(0,T)\otimes \mathcal{V}\text{.}%
\end{array}%
\right.  \label{6.37}
\end{equation}%
It is an easy matter to deal with (\ref{6.36}). In fact, fix $\xi \in 
\mathbb{R}^{N\times N}$ and consider the following cell problem: 
\begin{equation}
\left\{ 
\begin{array}{l}
\mathbf{\boldsymbol{\pi }}(\xi )\in \mathcal{B}_{AP}^{p}(\mathbb{R}_{\tau };%
\mathcal{B}_{\text{div}}^{1,p}): \\ 
\left[ \rho \frac{\partial \mathbf{\boldsymbol{\pi }}}{\partial \tau },%
\boldsymbol{w}\right] +\int_{\mathcal{K}}\widehat{b}\left\vert \xi +\partial 
\widehat{\mathbf{\boldsymbol{\pi }}}(\xi )\right\vert ^{p-2}(\xi +\partial 
\widehat{\mathbf{\boldsymbol{\pi }}}(\xi ))\cdot \partial \widehat{%
\boldsymbol{w}}d\beta \\ 
+\int_{\mathcal{K}}\widehat{a}(\xi +\partial \widehat{\mathbf{\boldsymbol{%
\pi }}}(\xi ))\cdot \partial \widehat{\boldsymbol{w}}d\beta =0\text{\ for
all }\boldsymbol{w}\in \mathcal{B}_{A_{\tau }}^{p}(\mathbb{R}_{\tau };%
\mathcal{B}_{\text{div}}^{1,p})\text{.}%
\end{array}%
\right.  \label{6.38}
\end{equation}%
Assume for a while that the solution to (\ref{6.38}) exists. Then, due to
the properties of the functions $a$ and $b$, the linear functional 
\begin{equation*}
\boldsymbol{w}\mapsto \int_{\mathcal{K}}\widehat{b}\left\vert \xi +\partial 
\widehat{\mathbf{\boldsymbol{\pi }}}(\xi )\right\vert ^{p-2}(\xi +\partial 
\widehat{\mathbf{\boldsymbol{\pi }}}(\xi ))\cdot \partial \widehat{%
\boldsymbol{w}}d\beta +\int_{\mathcal{K}}\widehat{a}(\xi +\partial \widehat{%
\mathbf{\boldsymbol{\pi }}}(\xi ))\cdot \partial \widehat{\boldsymbol{w}}%
d\beta
\end{equation*}%
defined on $\mathcal{B}_{AP}^{p}(\mathbb{R}_{\tau };\mathcal{B}_{\text{div}%
}^{1,p})$ is continuous, so that the linear functional $\boldsymbol{w}%
\mapsto \left[ \rho \frac{\partial \mathbf{\boldsymbol{\pi }}}{\partial \tau 
},\boldsymbol{w}\right] $, defined on $\mathcal{E}_{\text{div}}=\mathcal{D}%
_{AP}(\mathbb{R}_{\tau })\otimes \mathcal{V}_{\Div}$, is continuous when
endowing $\mathcal{E}_{\text{div}}$ with the $\mathcal{V}$-norm. From the
density of $\mathcal{E}_{\text{div}}$ in $\mathcal{B}_{AP}^{p}(\mathbb{R}%
_{\tau };\mathcal{B}_{\text{div}}^{1,p})$ (see Lemma \ref{l6.2}) we get
readily $\rho \frac{\partial \mathbf{\boldsymbol{\pi }}}{\partial \tau }\in 
\mathcal{B}_{AP}^{p^{\prime }}(\mathbb{R}_{\tau };(\mathcal{B}_{\text{div}%
}^{1,p})^{\prime })$, so that \textbf{$\boldsymbol{\pi }$} lies in $\mathcal{%
W}_{\text{div}}$. Since $\mathcal{E}_{\text{div}}$ is dense in $\mathcal{W}_{%
\text{div}}$, Eq. (\ref{6.38}) still holds for $\boldsymbol{w}\in \mathcal{W}%
_{\text{div}}$. But if \textbf{$\boldsymbol{\pi }$}$_{1}=$\textbf{$%
\boldsymbol{\pi }$}$_{1}(\xi )$ and \textbf{$\boldsymbol{\pi }$}$_{2}=$%
\textbf{$\boldsymbol{\pi }$}$_{2}(\xi )$ are two solutions of (\ref{6.38})
then, setting \textbf{$\boldsymbol{\pi }$}$=$\textbf{$\boldsymbol{\pi }$}$%
_{1}-$\textbf{$\boldsymbol{\pi }$}$_{2}$, 
\begin{equation*}
\left\{ 
\begin{array}{l}
\int_{\mathcal{K}}\widehat{a}\partial \widehat{\mathbf{\boldsymbol{\pi }}}%
\cdot \partial \widehat{\boldsymbol{w}}d\beta +\int_{\mathcal{K}}(\widehat{b}%
(\cdot ,\xi +\partial \widehat{\mathbf{\boldsymbol{\pi }}}_{1})-\widehat{b}%
(\cdot ,\xi +\partial \widehat{\mathbf{\boldsymbol{\pi }}}_{2}))\cdot
\partial \widehat{\boldsymbol{w}}d\beta =0 \\ 
\text{for all }\boldsymbol{w}\in \mathcal{W}_{\text{div}}\text{.}%
\end{array}%
\right.
\end{equation*}%
Taking the particular test function $\boldsymbol{w}=$\textbf{$\boldsymbol{%
\pi }$} and using the equality $\left[ \rho \frac{\partial \mathbf{%
\boldsymbol{\pi }}}{\partial \tau },\mathbf{\boldsymbol{\pi }}\right] =0$,
we are led to 
\begin{equation*}
\int_{\mathcal{K}}\widehat{a}\partial \widehat{\mathbf{\boldsymbol{\pi }}}%
\cdot \partial \widehat{\mathbf{\boldsymbol{\pi }}}d\beta +\int_{\Delta (A)}(%
\widehat{b}(\cdot ,\xi +\partial \widehat{\mathbf{\boldsymbol{\pi }}}_{1})-%
\widehat{b}(\cdot ,\xi +\partial \widehat{\mathbf{\boldsymbol{\pi }}}%
_{2}))\cdot \partial \widehat{\mathbf{\boldsymbol{\pi }}}d\beta =0,
\end{equation*}%
and using once again the properties of $a$ and $b$, we get 
\begin{equation*}
\nu _{0}\int_{\mathcal{K}}\left\vert \partial \widehat{\mathbf{\boldsymbol{%
\pi }}}\right\vert ^{2}d\beta +\nu _{1}\int_{\mathcal{K}}\left\vert \partial 
\widehat{\mathbf{\boldsymbol{\pi }}}\right\vert ^{p}d\beta =0,
\end{equation*}%
which gives $\partial \widehat{\boldsymbol{\pi }}=0$, or equivalently, $%
\overline{D}_{y}\boldsymbol{\pi }=0$. It then follows that $\boldsymbol{\pi }%
=0$ since it belong to $\mathcal{B}_{AP}^{p}(\mathbb{R}_{\tau };(\mathcal{B}%
_{\#AP}^{1,p}(\mathbb{R}_{y}^{N}))^{N})$. Returning to the existence issue,
Eq. (\ref{6.38}) admits a solution; see e.g., \cite[Proposition III.4.1]%
{Showalter}.

Now, choosing $\boldsymbol{\psi }_{1}=\varphi \otimes \boldsymbol{w}$ in (%
\ref{6.36}) with $\varphi \in \mathcal{C}_{0}^{\infty }(Q_{T})$ and $%
\boldsymbol{w}\in \lbrack \mathcal{D}_{AP}(\mathbb{R}_{\tau })\otimes 
\mathcal{V}_{\Div}]$, and using the arbitrariness of $\varphi$ we get the
following equation: 
\begin{equation}
-\left[ \rho \boldsymbol{u}_{1},\frac{\partial \boldsymbol{w}}{\partial \tau 
}\right] +\int_{\mathcal{K}}\widehat{a}\mathbb{D}\boldsymbol{u}\cdot
\partial \widehat{\boldsymbol{w}}d\beta +\int_{\mathcal{K}}\widehat{b}%
\left\vert \mathbb{D}\boldsymbol{u}\right\vert ^{p-2}\mathbb{D}\boldsymbol{u}%
\cdot \partial \widehat{\boldsymbol{w}}d\beta =0\text{\ for all }\boldsymbol{%
w}\in \mathcal{D}_{AP}(\mathbb{R}_{\tau })\otimes \mathcal{V}_{\Div}\text{.}
\label{6.39}
\end{equation}%
Coming back to (\ref{6.38}) we choose there $\xi =D\boldsymbol{u}_{0}(x,t)$
(for arbitrarily fixed $(x,t)\in Q_{T}$). Comparing the resulting equation
with (\ref{6.39}) and using the density of $\mathcal{D}_{AP}(\mathbb{R}%
_{\tau })\otimes \mathcal{V}_{\Div}$ in $\mathcal{B}_{AP}^{p}(\mathbb{R}%
_{\tau };\mathcal{B}_{\text{div}}^{1,p})$ we get by the uniqueness of the
solution of (\ref{6.38}) that $\boldsymbol{u}_{1}=\boldsymbol{\pi }(D%
\boldsymbol{u}_{0})$, where $\boldsymbol{\pi }(D\boldsymbol{u}_{0})$ stands
for the function $(x,t)\mapsto \boldsymbol{\pi }(D\boldsymbol{u}_{0}(x,t))$
considered as acting from $Q_{T}$ into $\mathcal{W}_{\text{div}}$. This
shows the uniqueness of the solution of (\ref{6.36}) and also allows us to
conclude that $\boldsymbol{u}_{1}$ lies in $L^{p}(Q_{T};\mathcal{W}_{\Div})$
so that $(\boldsymbol{u}_{0},\boldsymbol{u}_{1})\in \mathbb{F}_{0}^{1,p}$.
This concludes the proof of Theorem \ref{t6.3}.

In order to get the macroscopic homogenized problem, we set, for $\xi \in 
\mathbb{R}^{N\times N}$, 
\begin{equation*}
M(\xi )=\int_{\mathcal{K}}\widehat{b}\left\vert \xi +\partial \widehat{%
\boldsymbol{\pi }}(\xi )\right\vert ^{p-2}(\xi +\partial \widehat{%
\boldsymbol{\pi }}(\xi ))d\beta
\end{equation*}%
and 
\begin{equation*}
\mathsf{m}\xi =\int_{\mathcal{K}}\widehat{a}(\xi +\partial \widehat{%
\boldsymbol{\pi }}(\xi ))d\beta .
\end{equation*}%
Next, consider the following problem: Find $(\boldsymbol{u}_{0},q_{0})$
solution to 
\begin{equation}
\left\{ 
\begin{array}{l}
\mathfrak{M}_{y}(\rho )\frac{\partial \boldsymbol{u}_{0}}{\partial t}-\Div(%
\mathsf{m}D\boldsymbol{u}_{0})+M(D\boldsymbol{u}_{0}))+B(\boldsymbol{u}%
_{0})+\nabla q_{0}=\mathfrak{M}(\rho f(\cdot ,\boldsymbol{u}_{0})) \\ 
\Div\boldsymbol{u}_{0}=0\text{ in }Q_{T} \\ 
\boldsymbol{u}_{0}=0\text{ on }\partial Q\times (0,T) \\ 
\boldsymbol{u}_{0}(x,0)=\boldsymbol{u}^{_{0}}(x)\text{ in }Q.%
\end{array}%
\right.  \label{6.40}
\end{equation}%
Problem (\ref{6.40}) is the \emph{macroscopic homogenized problem}. In
contrast to (\ref{6.27}) it only involves the macroscopic limit $(%
\boldsymbol{u}_{0},q_{0})$ of the sequence $(\boldsymbol{u}_{\varepsilon
},q_{\varepsilon })$ of solutions to (\ref{0.1}). Thus it describes the
macroscopic behavior of the above-mentioned sequence as the heterogeneities
go to zero.

Returning to Problem (\ref{6.40}) we know that $\mathfrak{M}_{y}(\rho )$ is
a positive constant (see Assumption (\textbf{A5}) at the beginning of this
section). Moreover one may easily check that the functional $\boldsymbol{u}%
\mapsto \mathfrak{M}(\rho f(\cdot ,\boldsymbol{u}))$ maps $\mathbb{H}$ into $%
L^{2}(Q_{T})^{N}$, and is Lipschitz continuous (this is due to the
properties of $f$ and $\rho $). Thus, assuming the existence of the velocity 
$\boldsymbol{u}_{0}$ in the above problem, it comes by \cite[Proposition 5]%
{Simon2} that the pressure $q_{0}$ will exist. We can hence argue as in the
proof of Theorem \ref{t1.1} to get the existence of a solution to (\ref{6.40}%
). We may therefore state the macroscopic homogenization result.

\begin{theorem}
\label{t6.4}Assume that \textbf{(A1)-(A5)} hold. Let $1+\frac{2N}{N+2}\leq
p<\infty $. For each $\varepsilon >0$ let $(\boldsymbol{u}_{\varepsilon
},q_{\varepsilon })$ be a solution of \emph{(\ref{0.1})}. Then there exists
a subsequence of $(\boldsymbol{u}_{\varepsilon },q_{\varepsilon
})_{\varepsilon >0}$ (not relabeled) which converges strongly in $%
L^{2}(Q_{T})^{N}$ (with respect to the first component $\boldsymbol{u}%
_{\varepsilon }$) and weakly in $L^{p^{\prime }}(Q_{T})$ (with respect to $%
q_{\varepsilon }$) to the solution of \emph{(\ref{6.40})}. Moreover any
limit point in $L^{2}(Q_{T})^{N}\times L^{p^{\prime }}(Q_{T})$ (in the above
sense) of $(\boldsymbol{u}_{\varepsilon },q_{\varepsilon })_{\varepsilon >0}$
is a solution to Problem \emph{(\ref{6.40})}.
\end{theorem}

\begin{proof}
Considering (\ref{6.26}), by a density argument, this equation still holds
for $\boldsymbol{\Phi }$ in $L^{p}(0,T;W_{0}^{1,p}(Q)^{N})\times L^{p}(Q_{T};%
\mathcal{W}_{\text{div}})$. So substituting $\boldsymbol{u}_{1}=\boldsymbol{%
\pi }(D\boldsymbol{u}_{0})$ in (\ref{6.26}) and choosing the particular $%
\boldsymbol{\Phi }=(\boldsymbol{\psi }_{0},\boldsymbol{\psi }_{1})$ with $%
\boldsymbol{\psi }_{1}=0$ and $\boldsymbol{\psi }_{0}\in \mathcal{C}%
_{0}^{\infty }(Q_{T})^{N}$, we end up (using some obvious computations) with
the result.
\end{proof}

\end{document}